\date{August 29, 2023}   %

\title{Double Forms, Curvature Integrals  and \\  the Gauss--Bonnet Formula}
\author{Marc Troyanov\footnote{Institut  de Math\'{e}matiques EPFL, 1015 Lausanne, Switzerland;
 \  marc.troyanov@epfl.ch }}

\documentclass[11pt,a4paper]{article}
\usepackage[utf8]{inputenc}
\usepackage[T1]{fontenc}
\usepackage[english]{babel}
\usepackage{amsmath,amsfonts,amssymb,enumerate,color}
\usepackage{stmaryrd,MnSymbol} 
\usepackage{amsthm}
\usepackage{appendix} 
\usepackage[hidelinks]{hyperref}
\usepackage{helvet} 

\usepackage{makeidx}
\makeindex

\numberwithin{equation}{section}

\newtheorem{theorem}{\rm\bf Theorem}[section]
\newtheorem{proposition}[theorem]{\rm\bf Proposition}
\newtheorem{lemma}[theorem]{\rm\bf Lemma}
\newtheorem{corollary}[theorem]{\rm\bf Corollary}
\newtheorem{remark}[theorem]{\rm\bf Remark}
\newtheorem{definition}[theorem]{\rm\bf Definition}

\newtheorem{example}[theorem]{Example}

\DeclareMathOperator{\dvol}{dvol}
\DeclareMathOperator{\Vol}{Vol}
\DeclareMathOperator{\Pf}{Pf}

\DeclareMathOperator{\Ind}{Ind}
\DeclareMathOperator{\Ct}{C}

\newcommand{\lalpha}{\makebox{\Large\ensuremath{\alpha}}}

\newcommand{\R}{\mathbb{R}}

\newcommand{\bra}{\langle}
\newcommand{\ket}{\rangle}

\def \R {\mathbb R}

\def \RR {\mathcal{R}}

\begin{document}
\maketitle

\bigskip

\begin{abstract}
The Gauss--Bonnet Formula is a significant achievement in 19th century differential geometry for the case of surfaces and the 20th century cumulative work of  H. Hopf,  W. Fenchel, C. B. Allendoerfer, A. Weil and S.S. Chern for higher-dimensional Riemannian manifolds. It relates the Euler characteristic of a Riemannian manifold to a curvature integral over the manifold plus a somewhat enigmatic boundary term. In this paper, we revisit the formula using the formalism of double forms, a tool introduced by de Rham, and further developed by Kulkarni, Thorpe, and Gray. We explore the geometric nature of the boundary term and provide some examples and applications.

\medskip 

{\small	 {\textit{Keywords:  } Gauss--Bonnet Formula, Double Forms, Curvature Integrals.}}

\smallskip 

{\small  {AMS Subject Classification: 58A1, 53C20.}}
\end{abstract} \hspace{10pt}

\tableofcontents 

\section{Introduction}

The Gauss--Bonnet Formula for   Riemannian  surfaces is a remarkable achievement of  19th century differential geometry, resulting from the cumulating work of K. F. Gauss, P. O. Bonnet, J.P.  Binet, and W. von Dyck. The result is well-covered in a number of  textbooks on Riemannian geometry. 

\smallskip

Its extension to higher-dimensions is the result of the combined work of  H. Hopf, W. Fenchel, C. B. Allendoerfer, A. Weil, and S. S. Chern,  spanning the period from 1925 to 1945. The formula relates the Euler characteristic of a compact Riemannian manifold with a curvature integral over the manifold plus a boundary term involving  the curvature of the manifold and the second fundamental form over the boundary.

\smallskip

In his formulation of the  Gauss--Bonnet Formula,  Chern   uses Élie Cartan's moving frame method.  His proof is very direct  but it does not readily reveal the geometric nature of the boundary term. In this paper, we revisit and reframe the higher-dimensional Gauss--Bonnet Formula using double forms, a formalism introduced by G. de Rham and further developed by R. Kulkarni,  J. A. Thorpe, and A. Gray in the 1960s and 70s. Double forms offer a powerful tool for manipulating curvature identities  and provide  us with a different point of view on  the Gaus-Bonnet Formula.

\smallskip

Our objective in this paper  is then  to present the Gauss--Bonnet Formula within the framework of double forms. We provide details for  Chern's original proof and  give some  examples and applications.

\smallskip 

While we do  provide some historical perspective, this paper is not meant to be  a historical  essay.  
Further historical remarks and developments can be found in other sources such as \cite{Gottlieb1996},  \cite{Wu2008},  as well as Volumes 3 and 5 of M. Spivak's treatise \cite{Spivak}. See also the interviews \cite{Bourguignon},  \cite{Interview}. 

\smallskip

Due to space and expertise limitations, we do not explore alternative proofs (such as based on  heat kernel,  probabilistic or  PDEs methods) or broader topics such as the theory of characteristic classes,  the Hirzebruch signature theorem and more generally the 
Atiyah--Singer index theorem. Applications to   physics are also not discussed.

\subsubsection*{Organization of the Paper}

The rest of the paper is organized as follows: In Section 2 we discuss the origin of the problem behind the Gauss--Bonnet Formula, beginning with Hopf's famous "Curvatura Integra" problem. Section 3 introduces the notion of double forms on a manifold. These are tensors that provide a convenient algebraic framework for manipulating various geometric quantities in Riemannian geometry.

\smallskip 

Section 4 is dedicated to a fully detailed proof of the higher-dimensional Gauss--Bonnet Formula for compact, even-dimensional Riemannian manifolds with boundaries, following Chern's arguments. 
We also reinterpret the boundary term using double forms; the formula is stated in three different formulations  in \eqref{NewGBC}. Some examples and applications are discussed  in Section 5.

\smallskip

Section 6 contains various remarks and comments, and a brief summary on the notion of double factorial that is often used throughout the paper. Some exercises are suggested in Section 7; the reader is advised to take a look at them at some point. The paper concludes with an appendix that offers the necessary background in Riemannian geometry, including a detailed introduction to Cartan's moving frame methods.

As for prerequisites, we assume the reader is familiar  with basic Differential and Riemannian geometry at the graduate level, including tensor fields and differential forms. Although we have tried to write a clean exposé, the reader should be aware that a number of algebraic calculations are quite intricate, and some of the details are  left to the reader. 

\smallskip

Let us conclude this introduction with a word about our notation. The manifolds we work with are assumed to be smooth and oriented. Typically, they are denoted as $M$ or $N$, where $m = \dim(M)$ and $n = \dim(N)$. Throughout most of the discussion, $m$ will be presumed to be an even number, while $n$ may be either odd or even. 
The boundary $\partial M$ of the manifold $M$, if non empty, will be given the orientation induced by the outer normal. 
 

\section{An  Overview of Hopf's Problem on the Curvatura Integra}
\label{sec.hopf}

The 2-dimensional Gauss--Bonnet Formula, for a compact Riemannian surface $(S,g)$  with piecewise smooth boundary, is conventionally stated as follows:
$$
  2\pi \chi(S) = \int_S KdA + \int_{\partial S} k_g ds + \sum_{j} \alpha_j,
$$
where $K$ represents the curvature, $k_g$ denotes the geodesic curvature of its boundary,  $\alpha_j$ refers to the exterior angles on the boundary and $\chi(S)$ is the Euler characteristic of the surface. 

\medskip 

This formula  has been extended to the case of  higher-dimensional manifolds in a series of papers published between 1925 and 1945 by H.  Hopf, W. Fenchel, C. Allendoerfer, A. Weil and S. S.  Chern, as we now discuss. 
The first and fundamental breakthrough is due to Hopf, who proved  in 1925 that for a closed hypersurface $M$ in Euclidean space $\R^{m+1}$
the following holds if $m$ is even:\footnote{No such formula can hold if $m$ is odd, since in this case we always have $\chi(M) =0$. See however Corollary \ref{cor.normaldegree2}.}
\begin{equation}\label{GBHopf}
   \frac{1}{2} \Vol(S^m) \chi(M) = \int_M K\dvol_M. 
\end{equation}
Here $K$ is  the \textit{Gauss--Kronecker curvature}  \index{Gauss--Kronecker curvature}   
of the hypersurface, that is, the product of its principal curvatures. We shall call Equation \eqref{GBHopf} the \textit{Gauss--Bonnet--Hopf} Formula;  its right hand side is the \textit{total curvature}, or \textit{curvatura integra},  of the hypersurface.  \index{Total curvature (curvatura integra)}

\medskip 

To prove the Formula, Hopf generalized a result of H. Poincaré that relates the winding number of a vector field on the boundary of a planar  domain to the indices of its singularities. We recall the definition: given a continuous  vector field $X$ on a smooth manifold $M$ having an isolated zero at $p\in M$, and choosing a  local coordinate system in the neighborhood $U$ of $p$, one may identify $X$ with a mapping
$X : U \to\mathbb{R}^m$; for $\varepsilon > 0$ small enough, the following map is   well defined: 
$$
 \zeta : \partial B(p,\varepsilon) \to S^{m-1}, \quad  \zeta (q) = \frac{X_q}{\|X_q\|},
$$
where $B(p,\varepsilon) \subset M$ is the ball around $p$ with respect to some Riemannian metric on $M$.
The \emph{index of $X$} at the point $p$ is defined to be the degree of that map:
\[
 \Ind(X,p) = \deg\left(\left.\zeta\right|_{\partial B(p,\varepsilon)} \right).
\]
The index  depends only on the vector field and not on the chosen coordinate system or the auxiliary Riemannian metric. With  this in mind, we state the Poincaré--Hopf Theorem: 

\smallskip

\begin{theorem}[Poincaré--Hopf]\label{Theorem_Poincare_Hopf}   \index{Poincaré--Hopf's Theorem} 
Let $M$ be a smooth compact manifold with (possibly empty) boundary. Let  $X$ be a vector field  on $M$ that   is transverse to the boundary and outward pointing. If $X$ has isolated zeroes at $x_1, \ldots, x_k \in M\setminus \partial M$,
then the following equality holds:
\[
\sum_{i=1}^k \mathrm{Ind}(X,x_i) = \chi(M),
\]
where $\mathrm{Ind}(X,x_i)$ represents the index of $X$ at the zero $x_i$, and $\chi(M)$ denotes the Euler characteristic of $M$.
\end{theorem}

Hopf proved this result in \cite{Hopf1925,Hopf1926}, mentioning Poincaré for the 2-dimensional case and earlier partial results by  Brouwer and Hadamard in higher-dimensions. A very readable proof  can be found in \cite[page 35]{Milnor}. See also references  \cite{Bredon,GuilleminPollack,Madsen}  for the case of manifolds without boundary. 

\medskip

\begin{corollary}\label{cor.normaldegree}
Let $D$ be a smooth compact domain in $\mathbb{R}^m$.  Denote by $\nu : \partial D \to S^{m-1}$  the exterior pointing  Gauss map of the boundary. The degree of that map is equal to the Euler characteristic of the domain $D$, that is, 
$$
 \deg(\nu) =   \chi(D).
$$
\end{corollary}
We will give later  a more general version of this result, see Corollary \ref{cor.normaldegree2}.

\medskip 

\begin{proof}
Let us  choose a vector field $X$ on $D$ with isolated zeroes at $x_1, \dots, x_k \in D\setminus \partial D$ and such that $X$ is outward pointing at the boundary. 
Let us set   $D_{\varepsilon} = D \setminus \cup_{i=1}^k B(x_i, \varepsilon)\subset M$ for some $\varepsilon >0 $ small enough.
We  denote by  $\omega$   the volume form on $S^{m-1}$ and  by $f : D_{\varepsilon} \to S^{m-1}$ the map defined by $f(q) = X_q/\|X_q\|$. Since  $\omega$  is a closed form, we have
$$
 0 = \int_{D_{\varepsilon}} f^*(d\omega) =    \int_{\partial D_{\varepsilon}} f^*(\omega)
 = \int_{\partial D} f^*(\omega) - \sum_{i=1}^k \int_{\partial B(x_i, \varepsilon)} f^*(\omega).
$$
To complete the proof,  observe that 
$$
  \int_{\partial D} f^*(\omega) = \Vol(S^{m-1}) \deg(\left. f \right|_{\partial D})   \quad \text{and} \quad \sum_{i=1}^k \int_{\partial B(x_i, \varepsilon)} f^*(\omega) 
  = \Vol(S^{m-1})\sum_{i=1}^k \Ind(X,x_i).
$$
Moreover $\deg(\left. f \right|_{\partial D}) = \deg(\nu)$ because $\left. f \right|_{\partial D}$ is homotopic to the Gauss map.
\\
\end{proof}

\medskip 

We are now in position  to prove Hopf's formula \eqref{GBHopf}. Recall first that,   by the  Jordan–Brouwer separation  Theorem,  any closed  smooth  hypersurface $M \subset \mathbb{R}^{m+1}$ is the boundary of a compact smooth domain  with   
$D \subset \mathbb{R}^{m+1}$, see e.g. \cite[page 89]{GuilleminPollack}. By definition, the Gauss--Kronecker curvature $K$ of $M$ is the Jacobian of the Gauss map $\nu : M \to S^m$, i.e. we have
$$
   K \dvol_{M} = \nu^*\left(\omega\right),
$$
where $\omega = \dvol_{S^n}$ is the volume form of $S^m$. Using the previous corollary, we therefore have 
$$
 \int_M K \dvol_{M}  =   \deg(\nu) \int_{S^m} \omega = \chi (D)  \Vol(S^m),
$$
and  \eqref{GBHopf} follows now from the relation
$$
 \chi(M) = \chi (\partial D) = \frac{1}{2} \chi (D).
$$
\qed 

\medskip 

In a second paper \cite{Hopf1925a}, also published in 1925, Hopf established the validity of \eqref{GBHopf} for compact space forms, that is,  Riemannian manifolds with constant sectional curvature. A proof will also be given later in   \S \ref{sec.spaceform}.  Motivated by  these results, he asked  for an  intrinsic version of that result. 
Let us quote from  Allendoerfer's paper \cite{Allendoerfer}:

\begin{quotation}
\fontfamily{libertine}\selectfont 
\begin{minipage}{0.8\textwidth} 
{\small 
 H. Hopf has repeatedly pointed to the problem of deciding whether one can define a curvature scalar in an arbitrary Riemannian manifold of even dimension with the following properties: 
\begin{enumerate}[\   $\circ$]
\item In the special case of a hypersurface it becomes identical with the Gauss--Kronecker curvature mentioned above.
\item It  satisfies a formula of Gauss--Bonnet type, i.e.. its integral over a region of the manifold can be expressed by a curvature integral over the boundary.
\item The integral of this  curvature scalar is a topological invariant [in the case of a closed manifold].
\end{enumerate}
}
\end{minipage}
\end{quotation}

\medskip  

In short, Hopf is asking for the definition of a locally defined,  scalar function on the manifold, depending only on the metric tensor and 
whose integral is the Euler characteristic (plus a boundary term if $\partial M \neq \emptyset$). 
We will call this the \textit{Hopf Problem on Curvature Integra}.  \index{Hopf Problem on Curvatura Integra} 
A first observation is that, using Gauss' Equation relating the second fundamental form of a hypersurface $M$ in some Euclidean space to its curvature tensor:
$$
  R_{ijkl} = h_{ik}h_{jl}-h_{il}h_{jk}, 
$$
one can express the Gauss--Kronecker curvature $K(p)$ at any  point $p$ in $M$ as follows
\begin{equation} \label{RGKcurvature}
  K  = \frac{1}{m!!}   \sum_{\sigma\in S_{m}}   \sum_{\tau\in S_{m}} (-1)^{\sigma}(-1)^{\tau}
  R(e_{\sigma_1}, e_{\sigma_2},e_{\tau_1},e_{\tau_2})  \cdots  R(e_{\sigma_{m-1}}, e_{\sigma_m},e_{\tau_{m-1}},e_{\tau_m}),
\end{equation}
where $\{e_1, \dots, e_m\}$ is an orthonormal basis of the tangent space of $M$ at the point $p$. 
Here  $S_m$ is the permutation group on $m$ symbols and $(-1)^{\sigma}$ is the signature of  $\sigma \in S_m$. 
In the above formula, we have used the double factorial, which for an even number $m$ is defined as
$$
  m !! = m \cdot (m-2) \cdots  2 = 2^{ \frac{m}{2}}  \left(\tfrac{m}{2} \right)!
$$
we refer to \S  \ref{sec.volumeSd} for more on this notion. 
A proof  of \eqref{RGKcurvature} can be found on page 228 of the book \cite{Thorpe1979} by J.A. Thorpe, but another proof will be given in Proposition \ref{GK_intrinseque} below.
It is  obvious from \eqref{RGKcurvature} that the Gauss--Kronecker curvature  is intrinsic. 
 
\medskip 

The Hopf problem stated above has been   positively resolved in two papers,  \cite{Allendoerfer} and \cite{Fenchel},  published   independently in 1940,  by C. B. Allendoerfer  in the U.S.A, and W. Fenchel in Denmark. They proved that \eqref{RGKcurvature} holds for any  closed even-dimensional manifold  $M$ that can be isometrically embedded in some Euclidean space $\R^d$. Note that the embeddability condition is in fact not restrictive as was proved 16 years later by J. Nash in \cite{Nash1956}. 

\smallskip 

In both proofs, the authors consider the boundary of the  $\varepsilon$-tubular neighborhood of $M$, with $\varepsilon>0$ small enough; let us denote it by $N$,  and observe that it is a hypersurface in $\R^d$. Moreover,  the natural projection $\pi : N \to M$ is a fiber bundle with typical fiber the sphere $S^{d-1-m}$. 

\smallskip 

The scalar $K_p$, defined at any point $p$ of $M$  by integrating the Gauss--Kronecker curvature of  the hypersurface $N$ over the $\varepsilon$-sphere  $\pi^{-1}(p) \subset N$, is the generalized curvature function answering Hopf's Problem. This is verified by applying \eqref{GBHopf} to the hypersurface $N$ and computing.  Fenchel mentions that R.  Lipschitz and W. Killing already proved in the 1880s that $K_p$ depends solely on the Riemannian metric of $M$ and not on its embedding in Euclidean space. Alternatively, Allendoerfer used calculations from H. Weyl's 1938 paper on the tube formula \cite{Weyl} to establish that \eqref{RGKcurvature} still holds for the averaged Gauss--Kronecker curvature as defined.

\smallskip 

In short, the Hopf Problem for a closed submanifold $M \subset \R^d$ is resolved by applying the Gauss--Bonnet--Hopf Formula \eqref{GBHopf}  to a tubular neighborhood $N$ of the manifold, then integrating the Gauss--Kronecker curvature over the fibers of the fibration  $\pi : N \to M$, and proving that the resulting function is an intrinsic quantity attached to the Riemannian manifold.

\medskip 

The two aforementioned papers did not provide a complete solution to the Hopf Problem as stated earlier,  since the second requirement explicitly called for a version of the Gauss--Bonnet formula applicable to manifolds with boundary.
Also the condition that the manifold should be isometrically embedded in some Euclidean space was seen as a serious restriction (Nash's Theorem not being available at that time). The  joint paper \cite{AW} published in 1943 by Carl B. Allendoerfer and André Weil addresses both issues with the following method:  They consider a sufficiently fine triangulation of the manifold and use a result proved by  E. Cartan in 1927 that guarantees the possibility to \textit{locally} embed a Riemannian manifold isometrically into some $\R^d$, provided it is analytic.  Using Cartan's Theorem, one can isometrically embed each cell of the triangulation  in  some Euclidean space $\R^d$, then consider a tubular neighborhood  of that cell and repeat the previous construction. The result is a Gauss--Bonnet Formula, first for piecewise analytic  Riemannian polyhedra, and then, invoking an approximability result by Whitney,  for general  Riemannian polyhedra of class $\Ct^2$.

\medskip 

The result achieved by Allendoerfer and Weil indeed solved Hopf's problem, but it was not entirely satisfactory. The proof involved a high level of complexity and, furthermore, it relied on local embeddings in Euclidean space, which is clearly not the most natural approach for constructing an intrinsically defined generalized curvature function on a given Riemannian manifold. 

\bigskip 

According to the story, as told say in \cite{Wu2008}, Chern visited the Institute for Advanced Study in Princeton in August 1943. He was 31 years old. There,  he met André Weil, who discussed with him the Hopf problem and the need for an intrinsic proof. Within two weeks, Chern solved the problem using the moving frame method that had been recently introduced by Élie Cartan, and some innovative  ideas of his own. He published his proof in two papers in the Annals of Mathematics in 1944 and 1945. The first paper, titled \,  \textit{A Simple Intrinsic Proof of the Gauss--Bonnet Formula for Closed Riemannian Manifolds}, contains a short intrinsic proof of the Gauss--Bonnet Formula for closed even-dimensional Riemannian manifolds. The second paper, titled  \ \textit{On the Curvatura Integra in a Riemannian Manifold}, extends the investigation to manifolds of either odd or even dimension with boundary. 
The 1945 paper  also gives  additional formulas  expressing some characteristic classes of sphere bundles using  curvature integrals. Chern's work laid the foundation for the development of Chern--Weil theory, which relates characteristic classes to differential forms.

 
\section{Double Forms and their Geometric Applications}\label{Subsection_Algebra_of_Double Forms}

The notion of \textit{double form} \index{Double Form} on a manifold $N$ has been introduced by G. de Rham  who defined them as differential forms on a manifold $N$ with values in the Grassmann algebra of $TN$, see \cite[\S 7]{deRham}. Double forms have  been used in the 1960s and 70s as a convenient tool in the computation of various curvature formulae in Riemannian  geometry  by R. Kulkarni,  J. Thorpe and A. Gray.  In this section we introduce the algebra of double forms following the book  \cite{Gray} by Gray and we refer to \cite{Labbi} for further developments.

\medskip 

\subsection{Definition and examples}
 
\begin{enumerate}[(i)]
\item  Let $N$ be a smooth $n$-dimensional manifold.  A  \emph{double-form}  \index{double forms}
$A$ of type $(k,l)$ on $N$ is a smooth covariant tensor field  of degree $k+l$, 
which is alternating  in the first $k$ variables and also in the last $l$ variables. This condition can be expressed  as follows:
If $X_1,  \dots, X_k, Y_1, \dots, Y_l$ are $(k+l)$ vector fields on $N$ and $\sigma, \tau \in S_k$, then
$$
  A(X_1,  \dots, X_k; Y_1, \dots, Y_l) =  
  (-1)^{\sigma}(-1)^{\tau}A(X_{\sigma_1},  \dots, X_{\sigma_k}; Y_{\tau_1}, \dots, Y_{\tau_l}). 
$$

We denote by $\mathcal{D}^{k,l}(N)$ the  set of all $(k,l)$-double forms on $N$. This  is a module over the algebra  of smooth functions $C^{\infty}(N)$ that can be expressed as the tensor product 
\[
 \mathcal{D}^{k,l}(N) = \mathcal{A}^k(N) \otimes_{C^{\infty}(N)} \mathcal{A}^l(N),
\]
where $\mathcal{A}^k(N)$ is the usual space of   differential $k$-forms on $N$. The direct sum of these spaces is denoted as 
$$
 \mathcal{D}(N)  = \bigoplus_{k,l=0}^n  \mathcal{D}^{k,l}(N).
$$

\item A natural product,  denoted by  $\owedge$,  is defined on  $\mathcal{D}(N)$ as follows:   If $\psi_1=\alpha_1\otimes\beta_1 \in \mathcal{D}^{k,l}(N)$ and $\psi_2=\alpha_2\otimes\beta_2 \in \mathcal{D}^{p,q}(N)$ then
 $\psi_1 \owedge \psi_2$ is the following double-form of type $(k+p,l+q)$:
\[
 \psi_1\owedge\psi_2 = (\alpha_1\wedge\alpha_2)\otimes (\beta_1\wedge\beta_2) \in \mathcal{D}^{k+p,l+q}(N).
\]
The product is then extended to general elements of  $\mathcal{D}(N)$ by linearity.; it can also be defined by the formula:
\begin{align*}
&\left(\psi_1 \owedge \psi_2\right) (X_1,  \dots, X_{k+p}; Y_1, \dots, Y_{l+q})   =  
\\  &   
\frac{1}{k! \,  l! \, p!\, q!}   \sum_{\sigma\in S_{k+p}} \sum_{\tau\in S_{l+q}} (-1)^{\sigma}(-1)^{\tau} 
  \psi_1(X_{\sigma_1},  \dots, X_{\sigma_k}; Y_{\tau_1}, \dots, Y_{\tau_l}) 
\cdot 
 \psi_2(X_{\sigma_{k+1}},  \dots, X_{\sigma_{k+p}}; Y_{\tau_{l+1}}, \dots, Y_{\tau_{l+q}}) 
\end{align*}
We shall  call this operation the \textit{double wedge product}  \index{double wedge product}    of $\psi_1$ and $\psi_2$.

\item The pair $\left(\mathcal{D}(N), \owedge\right)$  is the  \emph{algebra of double-forms} on the manifold $N$, it is a  supercommutative  bigraded algebra, meaning that if $\psi_1 \in \mathcal{D}^{p,q}(N)$ and $\psi_2 \in \mathcal{D}^{r,s}(N)$, then $ \psi_1 \owedge \psi_2 \in  \mathcal{D}^{k+r,l+s}(N)$ and
$$
 \psi_1 \owedge \psi_2 = (-1)^{pr+qs}\psi_2 \owedge \psi_1.
$$
\end{enumerate}

\medskip 

\begin{remark} \rm 
The product in the algebra $\mathcal{D}(N)$ is simply denoted as $\psi_1\cdot\psi_2$ by various authors such as in  \cite{Kulkarni} and  \cite{Labbi}, and it is denoted by   $\psi_1\wedge\psi_2$ in the book \cite{Gray}.  
The notation $\owedge$ is used in \cite{Besse} for symmetric double forms of type $(1,1)$. We find it convenient to extend this notation  to the whole  algebra of double forms.
\end{remark}


\bigskip 

\textbf{Examples.} 
\begin{enumerate}[(a)]
\item The double forms of type $(0,0)$ are the smooth functions on $N$, that is, $\mathcal{D}^{0,0}(N) = C^{\infty}(N)$.
\item To any differential $k$-form $\alpha \in \mathcal{A}^k(N)$ we can associate the double forms $1\owedge \alpha \in \mathcal{D}^{0,1}(N)$
and $\alpha \owedge 1\in \mathcal{D}^{1,1}(N)$. This gives two natural embeddings  
of the exterior algebra into the algebra of double forms.
\item The double forms of type $(1,1)$ are the covariant tensor  fields of order $2$ on $M$. In particular the 
metric tensor $g$ of a Riemannian manifold is an element in $\mathcal{D}^{1,1}(N)$. 
\item If $N$ is a (cooriented) hypersurface in a Riemannian manifold $(M,g)$, then the second fundamental form $h$ of $N\subset M$ is another important  double forms of type $(1,1)$. 
\item The curvature tensor $R$ of a Riemannian manifold is a double form of type $(2,2)$. It is furthermore symmetric in the following sense:
$$
  R(X_1,X_2:Y_1,Y_2) =   R(Y_1,Y_2;X_1,X_2).
$$
\item The double wedge product of two symmetric  $(0,2)$ tensors fields $h',h'' \in \mathcal{D}^{1,1}(N)$ is usually called their  \emph{Kulkarni--Nomizu product}. It is explicitly given by
\begin{multline*}
 h_{1}  \owedge   h_{2} (X_{1},X_{2};Y_{1},Y_{2}) =  h_{1}(X_{1},Y_{1})h_{2}(X_{2},Y_{2}) -   h_{1}(X_{2},Y_{1})h_{2}(X_{1},Y_{2}) \\  + h_{1}(X_{2},Y_{2})h_{2}(X_{1},Y_{1}) - h_{1}(X_{1},Y_{2})h_{2}(X_{2},Y_{1}).
\end{multline*}

\item A covariant tensor field $S$ of degree $3$ can be seen as an element in $\mathcal{D}^{2,1}(N)$ if and only if it is skew symmetric in its first two variables: $S(X,Y;Z) + S(Y,X;Z) = 0$.
\item Given $S \in \mathcal{D}^{2,1}(N)$ and a $1$-form $\alpha$ seen as an element in $\mathcal{D}^{0,1}(N)$, the product $S \owedge \alpha \in \mathcal{D}^{2,2}(N)$ is given by
$$
 S \owedge \alpha (X,Y;Z,W) = S(X,Y;Z) \alpha (W) - S (X,Y;W)\alpha (Z).
$$
\item A smooth field $A\in \mathrm{End}(\Lambda^k TN)$ of endomorphisms of  the exterior power $\Lambda^kTN$ can be seen as the double form in $\mathcal{D}^{k,k}(N)$  defined by
$$
(X_1, \dots X_k ; Y_1, \dots Y_k) \mapsto \left\langle  X_1 \wedge X_2, \wedge \dots \wedge  X_k ,  A(Y_1 \wedge Y_2, \wedge \dots \wedge  Y_k)\right\rangle,
$$
where we have used the natural  scalar product $\langle \; , \; \rangle$ on $\Lambda^k TN$ induced by $g$.
\end{enumerate}

 \subsection{Transposition and Symmetric Double Forms}

The \textit{transpose} of a double form $\psi \in \mathcal{D}^{k,l}(N)$ is the double form $\psi^{\top} \in \mathcal{D}^{l,k}(N)$ defined by 
\[
   \psi^{\top}(X_1, \dots, X_l; Y_1, \dots, Y_k) = \psi(Y_1, \dots, Y_k; X_1, \dots, X_l).
\]
We trivially have  $({\psi^{\top}})^{\top} = \psi$, hence  transposition defines an involution of the  algebra of double forms Furthermore we have 
\[
  (\psi_1 \owedge \psi_2)^{\top} = \psi_1^{\top} \owedge \psi_2^{\top}.
\]
A double form $\psi \in \mathcal{D}^{k,k}(N)$ is termed \textit{symmetric} if $\psi^{\top} = \psi$. Note that the product of two symmetric double forms is itself a symmetric double form.

\medskip

The symmetric elements of the double algebra $\mathcal{D}(N)$ form a commutative subalgebra, denoted by Kulkarni as $\mathcal{C}$. He refers to it as the \textit{Ring of Curvature Structures}. Moreover, he considers sseveral  subalgebras of $\mathcal{C}$, characterized by the  Bianchi identities.

\subsection{The Contraction Operator}\label{Subsection_Contraction_Operators}

Given a  Riemannian metric $g$ on the manifold  $N$, and $1 \leq k,l \leq n = \dim(N)$, we inductively define the  \emph{contraction operator of order $p$:} 
\index{contraction (of a double form)}
$$
 \Ct^p :  \mathcal{D}^{k,l}(N) \to \mathcal{D}^{k-p,l-p}(N)
$$
for any integer $p$ such that   $0 \leq p \leq \min(k,l)$ 
as follows. For $\psi \in \mathcal{D}^{k,l}(N)$ we set $\Ct^0(\psi)=\psi$ and at any point $x\in N$ we set 
\[
 \Ct^p(\psi)(X_1,\ldots,X_{k-p};Y_1,\ldots,Y_{l-p}) = \sum_{i=1}^n\Ct^{p-1}(\psi)(X_1,\ldots,X_{k-p},e_i;Y_1,\ldots,Y_{l-p},e_i),
\]
for $ 1 \leq p \leq \min\{l,k\}$, where $e_1,\ldots,e_n \in T_xM$ is an   orthonormal basis and $X_i, Y_j \in T_xM$ are arbitrary tangent vectors. This operation is independent of the chosen orthonormal basis. 

\medskip

As an example, the first and second contractions of the curvature tensor $R$ of  a Riemannian manifold  are respectively its  Ricci and  scalar curvatures. 

\medskip

If $\psi \in  \mathcal{D}^{p,p}(N)$, its \textit{full  contraction} is defined to be $\Ct^p(\psi)$. This is a scalar quantity (i.e. an element of $C^{\infty}(M)$) and the following result is useful in computations:

\begin{proposition} \label{prop.contractnform}
Let $\{e_1, \dots, e_n\}$ be a positively oriented orthonormal moving frame defined in a domain $U$ of an oriented  $n$-dimensional Riemannian manifold  $(N,g)$, and let $\{\theta^1, \dots \theta^n\}$ be the dual coframe. 
Then the following holds: If  $\mu_1,\ldots,\mu_{p}$ and $\nu_1,\ldots,\nu_{p}$ are two lists of pairwise distinct indices in $\{1, \ldots, m\}$, then
\begin{align*}
\frac{1}{p!} \; \Ct^{p}\left(\theta^{\mu_1}\wedge \cdots \wedge\theta^{\mu_{p}} \otimes   \theta^{\nu_1}\wedge \cdots\wedge\theta^{\nu_{p}} \right) 
&=  \delta_{\mu_1,\ldots,\mu_{p}}^{\nu_1,\ldots,\nu_{p}}
\end{align*}
is equal to $+1$ or $-1$,  depending on whether ${\nu_1,\ldots,\nu_{p}}$ is an odd or even permutation of ${\mu_1,\ldots,\mu_{p}}$, and equal to zero otherwise.
\end{proposition}

\begin{proof}
By definition we have
\begin{align*}
 &\Ct^{p}\left(\theta^{\mu_1}\wedge \cdots \wedge\theta^{\mu_{p}} \otimes   \theta^{\nu_1}\wedge \cdots\wedge\theta^{\nu_{p}} \right) 
\\&  \qquad =
\sum_{i_1, \dots, i_m = 1}^{m}  \left( \theta^{\mu_1}\wedge \cdots \wedge\theta^{\mu_{p}} (e_{i_1}, \dots , e_{i_p}) \right)  \cdot  
\left( \theta^{\nu_1}\wedge \cdots \wedge\theta^{\nu_{p}} (e_{i_1}, \dots , e_{i_p})\right) 
\\&  \qquad = \sum_{i_1, \dots, i_m = 1}^{m}  \det \left( a^j_k\right) \cdot  \det  \left( b^j_k\right) 
\\&  \qquad = p!  \; \delta_{\mu_1,\ldots,\mu_{p}}^{\nu_1,\ldots,\nu_{p}},
\end{align*}
where we have used the notation 
$$
  a^j_k =  \left( \theta^{\mu_j }(e_{i_k})\right)  \quad \text{and} \quad   b^j_k =  \left( \theta^{\nu_j }(e_{i_k})\right).
$$
\end{proof}

\begin{corollary}\label{cir.contr}
If $\alpha, \beta \in \mathcal{A}^p(N)$, then 
$$
  \bra \alpha, \beta\ket = \frac{1}{p!}  \Ct^{p}\left(\alpha \otimes \beta\right),
$$
where $\bra \ , \ \ket$  is the natural scalar product defined on the Grassmann algebra (see  \eqref{scalprodext}).
\end{corollary}

\begin{proof}
This follows immediately  from the previous proposition and the fact that at any point $x$ of the manifold, the family 
$\{ \theta^{i_1}\wedge \cdots \wedge\theta^{i_{p}} \}_{i_1 < \cdots< i_p}$  is  orthonormal  in $\Lambda^pT_x^*N$.
\\
\end{proof}

In the next result we  gather  several useful contraction formulas:
\begin{proposition} \label{prop.contractnform2}
With the same notation as above, we have 
\begin{enumerate}[(1)]
\item Any  $\alpha \in \mathcal{A}^p(N)$ can be written as 
$$
 \alpha = \frac{1}{p!} \sum_{i_1 < \dots  < i_p = 1}^{n}  
  \Ct^n\left(\lalpha \otimes   \theta^{i_1} \wedge \cdots \wedge \theta^{i_p} \right) \theta^{i_1} \wedge \cdots \wedge \theta^{i_p}. 
$$

\item If  $\lalpha \in \mathcal{A}^n(N)$  is a differential of degree $n$, then for any permutation 
$\sigma \in S_n$ we have 
$$
 n! \;  {\lalpha} = (-1)^{\sigma}\Ct^n\left(\lalpha \otimes   \theta^{\sigma_1} \wedge \cdots \wedge \theta^{\sigma_m} \right) \dvol_{N,g}.
$$
\item For any  $\psi \in \mathcal{D}^{p,p}(N)$ we have 
\begin{align*}
  \Ct^{p+q}\left(\psi \owedge g^q\right) &=   \prod_{j=1}^{q} (p+j)(n-p+1-j)  \Ct^{p}\left(\psi\right)
\\ &= (q!)^2\binom{p+q}{q} \binom{n-p}{q} \Ct^{p}\left(\psi\right)
\end{align*}
for any integer $1 \leq q \leq (n-p)$.   

In particular, we have
\begin{equation}\label{eq.Cgq}
 \Ct^q(g^q) = \dfrac{n!q!}{(n-q)!} = \binom{n}{q}(q!)^2.
\end{equation}
\end{enumerate}
\end{proposition}

\begin{proof}
The first formula follows immediately  from the corollary and the expansion   
$$
 \alpha =  \sum_{i_1 < \dots  < i_p = 1}^{n}  
 \  \left\langle \lalpha ,\theta^{i_1} \wedge \cdots \wedge \theta^{i_p} \right\rangle \theta^{i_1} \wedge \cdots \wedge \theta^{i_p},
$$
and the second identity is a special case of the first.  

\smallskip

We now prove the third formula  for $q=1$. Using the first identity,  it suffices to consider the tensor
$\psi = \theta^{\mu_1}\wedge \cdots \wedge\theta^{\mu_{p}} \otimes \theta^{\mu_1}\wedge \cdots \wedge\theta^{\mu_{p}}$
where the $\mu_i$s are pairwise distinct. We then have
\begin{align*}
 \Ct^{p+1}\left(\psi \owedge g\right) &= 
\sum_{\nu =1}^n  \Ct^{p+1}\left(\theta^{\mu_1}\wedge \cdots \wedge\theta^{\mu_{p}}\wedge \theta^{\nu}\otimes 
\theta^{\mu_1}\wedge \cdots \wedge\theta^{\mu_{p}}\wedge \theta^{\nu} \right) \\ &=  (n-p) (p+1) !
\end{align*}
because there are exactly $(n-p)$  indices  $\nu \in \{1,\dots,n\} \setminus \{\mu_1,\ldots,\mu_{p}\}$.
 Since $\Ct^{p}\left(\psi\right) = p!$, we have proved that 
$$
  \Ct^{p+1}\left(\psi \owedge g\right) =  (n-p) (p+1)\Ct^{p}\left(\psi \right).
$$
The general case $q \geq 2$ easily follows by induction.
\end{proof}

\medskip

We can now prove that the Gauss--Kronecker Curvature of a hypersurface in $\R^{m+1}$ is  an intrinsic invariant of the underlying Riemannian manifold if $m$ is even. This observation can be seen as a generalization of Gauss' Theorema Egregium.

\begin{proposition}\label{GK_intrinseque}
The Gauss--Kronecker curvature of  a hypersurface  $M$  of even dimension $m$ in Euclidean space $\R^{m+1}$
can be written in terms of its Riemann curvature tensor as follows:
\begin{equation}\label{RGKcurvature2}
   K  =   \frac{2^{ {\frac{m}{2}}}}{(m!)^2} \Ct^m(R^{ {\frac{m}{2}}}) =   \frac{1}{m!!}  \sum_{\sigma\in S_{m}} \sum_{\tau\in S_{m}} (-1)^{\sigma}(-1)^{\tau} 
   R_{\sigma_{1}\sigma_{2}\tau_{1}\tau_{2}} \cdots  R_{\sigma_{m-1}\sigma_{m}\tau_{m-1}\tau_{m}}.
\end{equation}
\end{proposition}

In this formula, $R_{ijkl} = R(e_i,e_j;e_k,e_l)$, where $e_1,\dots, e_m$ is a given orthonormal frame. If the basis is not orthonormal, the right hand side must be divided by ${\det(g_{ij})}$.

\begin{proof}
From Lemma \ref{lem:formule_R} we know that 
$$
  R = \frac{1}{4}  \sum_{i,j,k,l = 1}^m R_{ijkl} \ \theta^i \wedge\theta^j \otimes  \theta^k\wedge\theta^l, 
$$
therefore 
$$
  R^{ {\frac{m}{2}}} = \frac{1}{4^{ {\frac{m}{2}}}}   \sum_{\mu_{1}, \dots ,\mu_{n}= 1}^m  \sum_{\nu_{1}, \dots ,\nu_{n}= 1}^m
 R_{\mu_{1}\mu_{2}\nu_{1}\nu_{2}} \cdots R_{\mu_{m-1}\mu_{m}\nu_{m-1}\nu_{m}}
\ \theta^{\mu_1} \wedge  \cdots \wedge \theta^{\mu_m} \otimes  \theta^{\nu_1}  \wedge \cdots \wedge \theta^{\nu_m},
$$
and using Proposition \ref{prop.contractnform}, we obtain
$$
 \frac{1}{m!} \Ct^m(R^{ {\frac{m}{2}}})  =  \frac{1}{4^{ {\frac{m}{2}}}}   \sum_{\sigma\in S_{m}} \sum_{\tau\in S_{m}} (-1)^{\sigma}(-1)^{\tau} 
 R_{\sigma_{1}\sigma_{2}\tau_{1}\tau_{2}} \cdots R_{\sigma_{m-1}\sigma_{m}\tau_{m-1}\tau_{m}}.
$$
The  second fundamental form $h$ of  $M\subset \R^{m+1}$  is a double form 
$h\in \mathcal{D}^{1,1}(M)$, and  by Gauss' Equation we have 
$$
  R = \frac{1}{2} h\owedge h = \frac{1}{2}   \sum_{i,j,k,l = 1}^m \  h_{ik}h_{jl}  \ \theta^i \wedge\theta^j \otimes  \theta^k\wedge\theta^l.  
$$
Repeating the previous argument we have 
\begin{align*}
 \frac{1}{m!} \Ct^m(R^{ {\frac{m}{2}}})  & =   \frac{1}{m!} \frac{1}{2^{ {\frac{m}{2}}}}  \Ct^m(h^m) 
\\&= 
\frac{1}{2^{ {\frac{m}{2}}}}   \sum_{\sigma\in S_{m}} \sum_{\tau\in S_{m}} (-1)^{\sigma}(-1)^{\tau} 
 h_{\sigma_{1}\tau_{1}}  h_{\sigma_{2}\tau_{2}}  \cdots  h_{\sigma_{m}\tau_{m}} 
\\&=\frac{m!}{2^{ {\frac{m}{2}}}}  \det \left(h_{ij}\right),
\end{align*}
and we conclude that 
\begin{align*}
 K = \det \left(h_{ij}\right) &=  \frac{2^{ {\frac{m}{2}}}}{(m!)^2} \Ct^m(R^{ {\frac{m}{2}}})
 \\&= \frac{1}{2^{ {\frac{m}{2}}}m!}  \sum_{\sigma\in S_{m}} \sum_{\tau\in S_{m}} (-1)^{\sigma}(-1)^{\tau} 
 R_{\sigma_{1}\sigma_{2}\tau_{1}\tau_{2}} \cdots R_{\sigma_{m-1}\sigma_{m}\tau_{m-1}\tau_{m}}.
\end{align*}
\end{proof}

\subsection{Integrating Double Forms}

Some interesting geometric  quantities can be obtained by considering a geometrically defined double form $\psi \in \mathcal{D}^{p,p}(M)$ and integrating its full contraction $\Ct^p(\psi)$ over the manifold. In particular, the 
following functional, which is defined on $\mathcal{D}^{2,2}(N)\times \mathcal{D}^{1,1}(N)$, will play an important role
in the sequel:
\begin{definition}\label{defQ} \rm 
Given a  Riemannian manifold  $(N,g)$ of dimension $n$, and double forms $A \in \mathcal{D}^{2,2}(N)$ and $b \in \mathcal{D}^{1,1}(N)$, we define for any integer $0 \leq k \leq \frac{n}{2}$
\begin{equation}\label{eq.defQ}
 \mathcal{Q}_k(A, b \mid N, g)  = \frac{1}{n! \, k! \, (n-2k) !} \int_N \Ct^n(A^k \owedge b^{n-2k}) \dvol_{g},
 \end{equation}
provided $\Ct^n(A^k \owedge b^{n-2k})$ is integrable on $N$. 

 \smallskip

The special cases where $k=0$ or $2k=n$ will also be denoted as
$$
 \mathcal{Q}_0(b \mid N, g)  =  \frac{1}{(n!)^2} \int_N \Ct^n( b^{n}) \dvol_{g} = \int_N  \det\left( b_{ij}\right) \dvol_{g},
$$
and, when  $n$ is even, 
$$
 \mathcal{Q}_{\frac{n}{2}} (A \mid N, g)  = \frac{1}{(n!)^2} \int_N \Ct^n(A^{\frac{n}{2}} ) \dvol_{g}.
$$
\end{definition}

The notation $\mathcal{Q}_k(A, b \mid N, g)$ is somewhat heavy but it precisely indicates what is being integrated on which manifold. When there is no risk of ambiguity, we will simplify it to   $\mathcal{Q}_k(A,b \mid N )$, or even  $\mathcal{Q}_k(A,b)$.

\smallskip

If  either $A$ or $b$ (or both) are curvature invariants of the manifold, the integral \eqref{eq.defQ} will be referred to as a \textit{curvature integral}. \index{curvature integral} 
We will need the special cases where $b$ is either the metric tensor $g$ or the second fundamental form $h$ in the case of a hypersurface, and $A$ is a linear combination of the curvature tensor $R$ and the double forms $g\owedge g$ or $h\owedge h$.

\medskip

\begin{example} \label{Ex.IndexQ}\rm 
As a first important example, the index of a smooth vector field $\xi$ on an $m$-dimensional  Riemannian manifold $(M,g)$ with an isolated singularity at $p\in M$ can be expressed as
$$
  \Ind(X,p) = \lim_{\varepsilon \to 0} \frac{1}{\Vol(S^{m-1})}\int_{S_\varepsilon} \det (L) \dvol_{S_\varepsilon},
$$
where $L : TM \to TM$ is defined as $L(Y) = \nabla_Y\zeta$ with $\zeta = \xi/\|\xi\|$.   Introducing the double form $\psi \in \mathcal{D}^{1,1}$ defined by  $\psi(X,Y) = \bra X,L(Y)\ket$, one can write the index as 
\begin{equation}\label{eQ.index}
  \Ind(X,p)  =   \lim_{\varepsilon \to 0} \tfrac{1}{\Vol(S^{m-1})}  \mathcal{Q}_0(\psi \mid {S_\varepsilon} , g). 
\end{equation}
\end{example}

\subsection{The Lipschitz--Killing Curvatures}

An important class of examples of curvature integrals is given by the \emph{Lipschitz--Killing Curvatures}, which we define now:

\begin{definition}\label{defLKcurvature}
The  \emph{total Lipschitz--Killing curvatures} of order $j$ of an $n$-dimensional Riemannian manifold $(N,g)$ are defined as $\mathcal{K}_j(N,g) =0$  if $j$ is odd and 
\begin{equation}\label{LKcurvature}
  \mathcal{K}_{2k}(N,g)  =  \mathcal{Q}_{k}(R,g \mid N),
\end{equation}
if $j = 2k \leq n$ is even.
\end{definition} \index{Lipschitz--Killing Curvatures}

\textbf{Remark.}  Using Proposition \ref{prop.contractnform},  we have 
$$\Ct^{n}\left(R^k \owedge g^{n-2k}\right)  = \dfrac{(n-2k)! \; n!}{(2k)!}  \Ct^{2k}\left(R^k\right),$$
therefore the  Lipschitz--Killing curvatures can also be written as 
\begin{equation}\label{LKcurvature2}
  \mathcal{K}_{2k}(N,g)  = \frac{1}{k! \, (2k)!} \int_N   \Ct^{2k}(R^k) \dvol_N,
\end{equation} 
which is the definition  given in  \cite{Gray}.

\medskip 

\begin{example}\label{Examples.TLK} \rm 
(a) The Lipschitz--Killing curvature of order $0$ is integrable if and only if $(N,g)$ has 
finite volume. In this case we have 
$$
  \mathcal{K}_0(N,g) = \text{Vol}_g(N).
$$
(b) The total Lipschitz--Killing curvature of order $2$ on a compact manifold  $(N,g)$ is
half the integral of its scalar curvature:
$$
\mathcal{K}_2(N,g) =\frac{1}{2} \int_{N} S_g  \dvol_N.
$$

\smallskip

(c)  If  $(N,g)$ is a compact space form with constant sectional curvature $a$, then using  $R = \frac{a}{2}g\owedge g$ we obtain 
\begin{equation}\label{Eq_TLK-curvature_space_form}
  \mathcal{K}_{2k}(N,g)   =   \frac{n! \ a^k  }{2^kk!(n-2k)!} \cdot\mathrm{Vol}(N,g).
\end{equation}
\end{example}


\medskip 

A notable application of the Lipschitz--Killing curvatures is illustrated by the celebrated H. Weyl  \textit{Tube Formula} proved in 1938:
\index{Weyl's tube formula}
\begin{theorem} 
If $N$ is a smooth closed submanifold of dimension $n$ in $\mathbb{R}^d$, then the volume of a (small enough) $\varepsilon$-tubular neighborhood
$
  U_{\varepsilon}(N) = \{x\in \mathbb{R}^d \mid \mathrm{dist}(x,N) \leq \varepsilon\}
$
is given by 
$$
\Vol\left(U_{\varepsilon}(N)\right) = \sum_{k=0}^{\lfloor\frac{n}{2}\rfloor} \alpha_{d,n,k}\, \mathcal{K}_{2k}(N) \, \varepsilon^{d-n+2k},
$$
where 
$$
\alpha_{d,n,k} = \frac{\pi^{\frac{d-n}{2}} }{2^{k+1}\Gamma \left(\frac{d-n}{2} +1+ k\right)}.
$$
Here $\lfloor{\frac{n}{2}}\rfloor$ denotes the nearest integer that is less than or equal to $\frac{n}{2}$. 
\end{theorem}

\smallskip

We refer to the original paper  \cite{Weyl} by H. Weyl or the book \cite{Gray} for a proof of the Tube Formula. 

\medskip 

The other important result is  the Gauss--Bonnet Formula itself, which, for a closed  Riemannian manifold $(M,g)$ of even dimension $m$,  can be written as
\begin{equation}\label{GaussBonnetLK}
  \mathcal{K}_{m}(M) = {(2\pi)^{\frac{m}{2}}}  \chi(M).
\end{equation}
As mentioned in Section \ref{sec.hopf} and explained in detail  in \cite[\S 5.5]{Gray}), Allendoerfer used some arguments and calculations from  Weyl's paper in a crucial step of his proof.  

\medskip

We will now give a new proof of the Gauss--Bonnet--Hopf Formula \eqref{GBHopf}, this time an intrinsic one,  as a special case of the Gauss--Bonnet Formula   \eqref{GaussBonnetLK}. To this aim, it  will be convenient to introduce the following definition: 
\begin{definition}\label{defRR} \index{Gauss--Bonnet Density}
The  \emph{Gauss--Bonnet density} of the Riemannian manifold $(M,g)$ of even dimension $m$ is the function $\RR : M \to \R$ defined as
\begin{equation}\label{eq.defRR}
  \RR = \frac {1}{m!\left(\frac{m}{2}\right)! } \Ct^m(R^{ {\frac{m}{2}}}).
\end{equation}
\end{definition}
Observe that the top Lipschitz--Killing curvature is the integral of the Gauss--Bonnet density:
\begin{equation}\label{intRR}
  \mathcal{K}_{m}(M) = \int_M \RR \dvol_g. 
\end{equation}
Note also that the  proof of Proposition \ref{GK_intrinseque}  shows that the Gauss--Bonnet density of an oriented even-dimensional hypersurface in Euclidean space is equal, up to a  multiplicative constant, to the Gauss--Kronecker curvature. More precisely we have  
\begin{equation}\label{eq.RisK}
  K =    \frac{2^{ {\frac{m}{2}}}}{(m!)^2} \Ct^m(R^{ {\frac{m}{2}}}) =  
\frac {m!!}{ m!} \, \RR. 
\end{equation}
Assuming the Gauss--Bonnet Formula \eqref{GaussBonnetLK} has been established, we now easily deduce \eqref{GBHopf}: 
\begin{equation}\label{GBHopf2}
 \int_M K d\mathrm{vol}_N =   \frac {{2}^{\frac{m}{2}} \, \left(\frac{m}{2}\right)!}{ m!} \, \int_M\RR  \dvol_M
 =    \frac {{2}^{\frac{m}{2}} \, \left(\frac{m}{2}\right)!}{ m!} \, (2\pi)^{\frac{m}{2}}\chi(M) = 
  \frac{1}{2} \Vol(S^m) \chi(M),
\end{equation}
because the volume of an even-dimensional sphere is   
$$
 \Vol(S^m) =   {2^{m+1} \pi^{\frac{m}{2}} \, \frac{ \left(\frac{m}{2}\right)!}{m!}  }.
$$

The proof of \eqref{GaussBonnetLK} will be given in next section. 

\medskip 

\begin{remark} \rm
The Gauss--Bonnet density $\RR$ that we have just defined is also referred to as the \textit{Gauss--Bonnet Integrand}.
Fenchel proposed calling it the \textit{Lipschitz--Killing curvature}, and some authors do adopt that terminology. However, we find it problematic considering our Definition \ref{defLKcurvature}. Note  that $\RR$ is precisely the intrinsically defined curvature function  that was asked for in Hopf's Curvatura Integra problem (up to a   possible multiplicative constant).
\end{remark}

\medskip 

\begin{remark} \rm
The Gauss--Bonnet density, as well as its integral, also made an appearance in the notable 1938 paper \cite{Lanczos} by C. Lanczos. His work's context was the formulation of the fundamental laws of nature based on the principle of least action. Lanczos observed that the integral $\int \RR$ remains invariant under variations of the metric $g$, establishing it as a differential invariant of the manifold. He concluded that, within the framework of field theory, the Gauss--Bonnet density doesn't bring any additional information. Lanczos' work was carried out independently of the concurrent efforts by mathematicians addressing the Hopf problem and the Gauss--Bonnet formula.
\end{remark}

\section{Chern's Insight and the Tao of Gauss--Bonnet}
For a compact even-dimensional manifold with boundary $(M,g)$, one needs to add a boundary term to \eqref{GaussBonnetLK}. 
This  additional term is a sum of curvature integrals involving the curvature tensor $R$ of $(M,g)$ and the second fundamental form $h$ of the boundary $\partial M$. The Gauss--Bonnet Formula  takes the following form:
\begin{equation*}
   (2\pi)^{\frac{m}{2}} \chi(M) = \mathcal{K}_{m}(M)  + \sum_{k=0}^{\frac{m}{2}-1} b_{m,k} \, \mathcal{Q}_k({R},h \mid \partial M, {g}),
\end{equation*}
where the coefficient $b_{m,k}$ will be identified in due time.  We will also obtain a similar formula, where the tensor  $R$ is being replaced by  the intrinsic curvature $\bar{R}$ of the boundary.

\medskip

Our goal in this section is  to  present a comprehensive and intrinsic proof of the above formula.  Our approach closely follows the works of Chern \cite{Chern1944,Chern1945}, with a few minor adjustments, which  are discussed  in \S \ref{sec.compchern}. In  what follows, the formalism of moving frames is extensively used; we refer the reader to the appendix, which contains all the necessary definitions and formulas. Also, a number of expressions below involve double factorials, see \S \ref{sec.volumeSd} for the definition ans some basic properties.

\subsection{The Emergence of the Pfaffian}

The ingredients in Chern's proof are differential forms and specifically E. Cartan's moving frame method. Referring to \eqref{intRR}, we see that one needs to closely consider the differential form 
$$
  \RR \dvol_g,
$$
where $\RR$ is the Gauss--Bonnet density defined in \eqref{defRR}. 
Invoking Lemma \ref{lem:formule_R}, we see that the double form $R^m \in \mathcal{D}^{m,m}(M)$ can be written as
\[
R^{ {\frac{m}{2}}} = \frac{1}{2^{ {\frac{m}{2}}}} \sum_{\mu_1,\dots ,\mu_m=1}^m {\Omega}_{\mu_1, \mu_2} \wedge \ldots \wedge {\Omega}_{\mu_{m-1},\mu_{m}}
\otimes \theta^{\mu_1} \wedge \cdots \wedge \theta^{\mu_m},
\]
and Proposition 5.4 implies
\[
\frac{1}{m!} \Ct^m(R^{ {\frac{m}{2}}}) \theta^{1} \wedge \cdots \wedge \theta^{m} =
\frac{1}{2^{ {\frac{m}{2}}}} \sum_{\sigma\in S_{m}} (-1)^{\sigma} {\Omega}_{\sigma_1, \sigma_2} \wedge \ldots \wedge {\Omega}_{\sigma_{m-1},\sigma_{m}}.
\]
This relation  can be written as
\begin{equation}\label{eq.prepfaff}
\RR \dvol_{g} = \Pf(\Omega),
\end{equation}
where   the $m$-form
\[
\Pf(\Omega) =  \frac{1}{m!!}  \sum_{\sigma\in S_{m}} (-1)^{\sigma} {\Omega}_{\sigma_1, \sigma_2} \wedge \ldots \wedge {\Omega}_{\sigma_{m-1},\sigma_{m}}
\]
is the \textit{Pfaffian}\footnote{The \textit{Pfaffian} of a skew-symmetric matrix $A = (a_{ij})\in M_m(\R)$ is defined for $m$ even to be 
$\Pf(A) = \frac{1}{m!!}  \sum_{\sigma\in S_{m}} (-1)^{\sigma} \prod_{i=1}^{ {\frac{m}{2}}}a_{\sigma_{2i-1},\sigma_i}$. The notion was introduced by A. Cayley.} of the curvature form. Observe that \eqref{eq.prepfaff} implies that the Pfaffian $\Pf(\Omega)$ does not depend on the chosen moving frame.
\index{Pfaffian}

\medskip

\textbf{Examples.}
Recall that in dimension $2$, there is only one non-zero curvature form (up to sign), namely $\Omega_{12} = -\Omega_{21}$. Therefore, for a Riemannian surface $(S, g)$, the Pfaffian of its curvature form reduces to:
\[
\Pf(\Omega) = \frac{1}{2}(\Omega_{12} - \Omega_{21}) = \Omega_{12} = \frac{1}{2}R_{12ij}\theta^i \wedge \theta^j = R_{1212}\dvol_S = K\dvol_S,
\]
where $K$ is the Gauss curvature of the surface $S$.

If $M$ is a 4-dimensional manifold, then its Pfaffian is given by:
\[
\Pf(\Omega) = \frac{1}{8} \sum_{\sigma\in S_{4}} (-1)^{\sigma} {\Omega}_{\sigma_1, \sigma_2} \wedge {\Omega}_{\sigma_{3},\sigma_{4}}
\]
is a sum of $24$ terms. But using  the symmetries $ \Omega_{ab} = - \Omega_{bc}$ and $\Omega_{ab}\wedge \Omega_{cd} = \Omega_{cd} \wedge \Omega_{ab}$, it can be reduced to 
\[
 \Pf(\Omega) = {\Omega}_{1,2} \wedge {\Omega}_{3,4} + {\Omega}_{2,3} \wedge {\Omega}_{1,4} + {\Omega}_{3,1} \wedge {\Omega}_{2,4}.
\]
See exercise 2 in \S \ref{sec.exercices} for a generalization  in higher dimension. 

\subsection{A Tale of Exactness, featuring a Vector Field}

The next result lies at the core of Chern's proof of the Gauss--Bonnet Formula. It states  that when the manifold carries a non-vanishing vector field, the Pfaffian of the curvature form is exact.

\begin{proposition}[Chern]\label{Primitive}
 Let $(M,g)$ be an oriented Riemannian  manifold of even dimension $m$. If $M$ admits an everywhere  non vanishing vector field $\xi$,  then $\Pf(\Omega)$ is an exact differential form.
\end{proposition}

We can rephrase this Proposition by expressing that the Gauss--Bonnet density $\mathcal{R}$ is a divergence. An immediate consequence is  the following corollary, which is a special case of the Gauss--Bonnet Formula:
\begin{corollary}
 Let $M$ be a compact oriented  manifold without boundary of even dimension $m$. If $M$ admits an everywhere  non vanishing vector field $\xi$,  then the top global Lipshitz--Killing  curvature of $M$ vanishes.
\end{corollary}
Indeed, by Stoke's Theorem we have $\displaystyle \mathcal{K}_m(M,g) = \int_M \RR  \dvol_{{g}} = 0$.

\medskip 

The rest of his section is devoted to proving the above  Proposition.

\medskip 

\textbf{Proof.}  
Following Chern's ideas, we use the non-vanishing vector field $\xi$ to construct an $(m-1)$-form $\Phi$ that satisfies $\Pf(\Omega) = d\Phi$. To achieve this, we first  introduce the  double form $h = h_\zeta \in \mathcal{D}^{1,1}(M)$ of type $(1,1)$ on $M$, defined as 
\begin{equation}\label{eq.defh}
  h(X,Y) =   \bra  \nabla_X \zeta, Y\ket,
\end{equation}
where $\zeta$ is the unit vector field   $\zeta = \xi/|\xi|$.
Next we define a collection of $(m-1)$-forms $\Phi_k$, for $0  \leq k <  \frac{m}{2}$, as follows:
\begin{equation}\label{dePhik}
\Phi_k =   - \frac{2^k}{(m-1)!} \cdot \Ct^{m-1} \left( R^k \owedge h^{m-1-2k} \right) \varpi,
\end{equation}
where $\varpi =   \varpi_{\zeta}$ is defined in \eqref{def.varpi}. Additionally, we set  $\Phi_{-1} = 0$ for convenience.

\medskip

We will now build a primitive of $\Pf(\Omega)$   as a linear combination of the  $\Phi_k$s. To perform the necessary computations, we choose a positively oriented moving frame $e_1, \dots, e_m$ such that $e_m = \zeta$. This is always possible  in a neighborhood  $U\subset M$ of any point. Referring to Definition \ref{defomegai}, we  then have  for any $1 \leq i, j \leq m-1$ 
$$
  h(e_i,e_j) =   g(\nabla_{e_i}e_m, e_j) =  \omega_{jm}(e_i),
$$
that is, 
\begin{equation}\label{eq.homega}
 h = \sum_{j=1}^{m-1} \omega_{jm}\otimes \theta^j.
\end{equation}
We will now use some formulas established in the appendix.  The first one is \eqref{normalvarpi}, which tells us that  
$$
  \varpi = - \theta^{1} \wedge \cdots \wedge \theta^{m-1}, 
$$
because  $m= \dim(M)$ is even. The second formula  is stated in  Lemma \ref{lem:formule_R}:
$$
 R =  \frac{1}{2} \sum_{i,j=1}^m \Omega_{ij} \otimes \theta^i \wedge \theta^j.
$$
We therefore have 
\begin{align*}
&  R^k \owedge h^{m-1-2k}  =
\\ &
  \frac{1}{2^k} \sum_{\mu_1,\ldots,\mu_{m-1} = 1}^{m-1}  
  \Omega_{\mu_1,\mu_2} \wedge \dots \wedge \Omega_{\mu_{2k-1}\mu_{2k}} \wedge  
\omega_{\mu_{2k+1},m}\wedge \cdots \wedge \omega_{\mu_{m-1},m} \otimes  
\theta^{\mu_1} \wedge \cdots \wedge \theta^{\mu_{m-1}}.
\end{align*}
Applying now Proposition \ref{prop.contractnform2}, we obtain the following local expression for the $(m-1)$-form $\Phi_k$: 
\begin{equation}\label{ChernPhik}
 \Phi_k =  \sum_{\sigma\in S_{m-1}} (-1)^{\sigma} {\Omega}_{\sigma_1, \sigma_2} \wedge\ldots\wedge {\Omega}_{\sigma_{2k-1},\sigma_{2k}} \wedge\omega_{\sigma_{2k+1},m}\wedge\ldots\wedge\omega_{\sigma_{m-1},m}.
\end{equation}
This is  how the form $\Phi_k$  appears in Chern's paper  \cite{Chern1944,Chern1945},  and we will now  closely  follow Chern's   computation of  its  exterior derivative.
 Because a 2-form commutes with any other differential form, we can write the exterior derivative of $\Phi_k$ as 
\begin{align*}
 d\Phi_k & = k\sum_{\sigma\in S_{m-1}} (-1)^{\sigma} d{\Omega}_{\sigma_1, \sigma_2} \wedge \Omega_{\sigma_3, \sigma_4} \wedge \ldots\wedge {\Omega}_{\sigma_{2k-1},\sigma_{2k}} \wedge\omega_{\sigma_{2k+1},m}\wedge\ldots\wedge\omega_{\sigma_{m-1},m}
\\ &
 + (m-2k-1)  \sum_{\sigma\in S_{m-1}} (-1)^{\sigma} {\Omega}_{\sigma_1, \sigma_2} \wedge\ldots\wedge {\Omega}_{\sigma_{2k-1},\sigma_{2k}} \wedge d\omega_{\sigma_{2k+1},m}\wedge \omega_{\sigma_{2k+2},m}\wedge \ldots\wedge\omega_{\sigma_{m-1},m}.
\end{align*}
We next use the second Structure Equation  \eqref{StructuresEq} and the  Bianchi identity \eqref{Bianchi1et2}. These identities allow us to rewrite  $d\Omega_{\sigma_{1},\sigma_{2}}$ and $d\omega_{\sigma_{k+1},m}$  as 
$$
 d\omega_{\sigma_{k+1},m} =  \Omega_{\sigma_{k+1},m} - \sum_{s=1}^m \omega_{\sigma_{k+1},s} \wedge \omega_{s,m}, 
$$
and 
$$
  d\Omega_{\sigma_{1},\sigma_{2}}  =  \sum_{s=1}^m \left(\Omega_{\sigma_{1},s}\wedge \omega_{s,\sigma_{2}} - \omega_{\sigma_{1},s} \wedge  \Omega_{s,\sigma_{2}}\right)
=
 \sum_{s=1}^m \left(  \Omega_{\sigma_{2},s} \wedge \omega_{\sigma_{1},s} - \Omega_{\sigma_{1},s}\wedge \omega_{\sigma_{2},s}  \right). 
$$
Carefully computing,\footnote{See \cite[Lemma 3.14]{Marcone} for a detailed explanation.} one then  finds that 
$$
	 d\Phi_k = \frac{m-2k-1}{2k+2}\Psi_k + \Psi_{k-1}  + \Xi_k,
$$
where $\Psi_k$ is the $m$-form
\begin{align*}
	\Psi_k = (-1)^{k+1} (2k+2) \cdot \sum_{\sigma\in S_{m-1}} (-1)^{\sigma} {\Omega}_{\sigma_1, \sigma_2} \wedge\ldots\wedge {\Omega}_{\sigma_{2k-1},\sigma_{2k}} \wedge\Omega_{\sigma_{2k+1},m} \wedge\omega_{\sigma_{2k+2},m},  \wedge\ldots\wedge\omega_{\sigma_{m-1},m}
\end{align*}
and $\Xi_k$ is a linear combinations of products of curvature and connection forms containing some $\omega_{ij}$ with $1 \leq i,j \leq (m-1)$.
The forms $\Psi_k$ are defined for $0 \leq k \leq \frac{m}{2}-1$, but  it is convenient to also define $\Psi_{k-1} = 0$. 
 
\medskip

A key fact about the previous formula,  observed by Chern, is that we must have  $\Xi_k = 0$.  The reason is that the forms $\Phi_k$, and hence $d\Phi_k$, do not depend on the chosen moving frame $e_1, \dots, e_m$, except for the constraint $e_m = \zeta$. Lemma \ref{lem.chgframeomega} implies therefore that all terms in $d\Phi_k$  containing some $\omega_{ij}$ with $1 \leq i,j \leq (m-1)$ must cancel each other. 
We have thus established the following relation:
\begin{equation}\label{chernidentity}
	d\Phi_k =  \frac{m-2k-1}{2(k+1)}\Psi_k  + \Psi_{k-1}.
\end{equation}
This identity gives us $\Psi_k$ as a linear combination of the $d\Phi_r$s. Let us indeed write
\begin{equation}\label{chernidentity2a}
  \Psi_k = \sum_{r=0}^k  \lambda_{m,k,r} \, d\Phi_r,
\end{equation}
where the unknown coefficients $\lambda_{m,k,r}$ are to be determined. From  \eqref{chernidentity} we have
$$
  d\Phi_k = \sum_{r=0}^k  \left(\frac{m-2k-1}{2(k+1)} \lambda_{m,k,r}+ \lambda_{m,k-1,r} \right) d\Phi_r. 
$$
The condition on  $\lambda_{m,k,r}$ is then
$$
   \frac{m-2k-1}{2(k+1)} \lambda_{m,k,r} + \lambda_{m,k-1,r} = \delta_{k,r},
 \ \text{ with   $\lambda_{m,k,r}= 0$ for $r > k$,} 
$$
which is equivalent  to 
\begin{equation}\label{formulelambda}
 \lambda_{m,k,r} =  (-1)^{k-r}\prod_{j=r}^k \frac{2j+2}{m-2j-1}.
\end{equation}
We have established that each $\Psi_k$ is an exact form. Considering the particular case $k= \frac{m}{2}-1$, we have 
\begin{align*}
 \Psi_{\frac{m}{2}-1} &=  
 	m\cdot \sum_{\sigma\in S_{m-1}} (-1)^{\sigma} {\Omega}_{\sigma_1, \sigma_2} \wedge\ldots\wedge {\Omega}_{\sigma_{m-1},\sigma_{m}} 
\\& =\sum_{\sigma\in S_{m}} (-1)^{\sigma} {\Omega}_{\sigma_1, \sigma_2} \wedge\ldots\wedge {\Omega}_{\sigma_{m-1},\sigma_{m}}
\\& = m!!  \, \Pf(\Omega),
\end{align*}
and we conclude that $\Pf(\Omega)$ is an exact form.

\qed

\bigskip 

\textbf{Remarks  1.}  The differential forms $\Psi_k$ have been defined using the connection and curvature forms of a particular moving frame, 
however Formula \eqref{chernidentity2a} implies that the  $\Psi_k$ do not depend on the chosen frame, provided $e_m = \zeta$. 
 
\medskip

\textbf{2.}  The  proof gives us an explicit formula for the primitive of the Pfaffian. We have 
\begin{equation}\label{FinalPhi}
   \Pf(\Omega) = d\Phi , \quad \text{with} \quad  \Phi  =   \sum_{r=0}^{\frac{m}{2}-1}  a_{m,r}  \Phi_r,
\end{equation}
where  the $a_{m,r}$ can be computed   from \eqref{formulelambda}:
\begin{equation}\label{Calcula1}
  a_{m,r}  =  \frac{1}{m!!}  \,  
\prod_{j=r}^{\frac{m}{2}-1}\frac{2j+2}{m-2j-1}
 =
 {\frac {1}{(2r)!! \, \left( m-2\,r-1 \right) !!}}.
\end{equation}
The coefficients  $a_{m,r}$  can be written in several ways, which all appear in the literature,. For instance\footnote{These identities hold for $m$ even and $r< {\frac{m}{2}}$.}
\begin{equation}\label{Calculaa2}
 a_{m,r} =  {\frac { 1 }{{2}^{r}r!\, \left( m-2\,r-1 \right) !!}}
= 
  \frac{2^{\frac{m}{2}-2r}\left( \frac{m}{2}-r\right)!}{  (m-2r)! \,  r!}
=
 {\frac {\sqrt {\pi}}{2^{\frac{m}{2}}\,  r!\,\Gamma \left( \frac{m-2\,r+1 }{2} \right) }}.
\end{equation}

\bigskip 
 
\subsection{Harvesting the ripe  Gauss--Bonnet  Fruit} 
\label{sec.Chern-Gauss--Bonnet_Bdy}
With the hard work behind us, we are now in a position to state and prove the Gauss--Bonnet Formula, as originally formulated by Chern (see Equation (19) in \cite{Chern1945}). 
\index{Gauss--Bonnet--Chern Formula}
\begin{theorem}\label{th.CGB}
The Euler characteristic of a  compact Riemannian manifold $(M,g)$ of even dimension $m$ with (possibly empty) boundary has the following integral representation:
$$
  (2\pi)^{\frac{m}{2}}  \chi(M) = \int_{M} \Pf(\Omega) - \int_{\partial M} \Phi.
$$
\end{theorem}

\begin{proof}
For the proof, we choose  a smooth vector field $\xi$ on $M$ with the following properties:
\begin{enumerate}[(i)]
\item $\xi$ has only isolated singularities at  $x_1, \dots, x_q \in M\setminus \partial M$.
\item At any boundary point,   $\xi$ is non vanishing, outward pointing  and orthogonal to $\partial M$.
\end{enumerate}
On the manifold $M' = M\setminus \{x_1, \dots, x_q\}$ we define the normalized vector field $\zeta = \xi/\|\xi\|$ and the double form   $h(X,Y) =  \bra X, \nabla_Y \zeta\ket$.   

\smallskip  

Let us then  denote by $B_i(\varepsilon)$ the open ball of radius $\varepsilon$ centered at $x_i$. We assume $\varepsilon > 0$ small enough so that the  balls $B_i(\varepsilon)$ are pairwise disjoint and have  empty intersection with $\partial M$,  and we set
$
 M_\varepsilon = M \setminus \cup_{i=1}^q B_i(\varepsilon).
$
The unit vector field $\zeta = \xi/|\xi|$ is well defined on $M_{\varepsilon}$ and we have by Stokes Theorem:
$$
 \int_{M_\varepsilon} \Pf(\Omega) =  \int_{\partial M} \Phi  \  -  \  
 \sum_{i=1}^q  \int_{\partial B_i(\varepsilon)} \Phi,
$$
where the form $\Phi$ has been defined in \eqref{FinalPhi},  see also \eqref{dePhik}. Chern  observes then  that
$$
 \lim_{\varepsilon \to 0} \int_{\partial B_i(\varepsilon)} \Phi_k = 0
$$
for $k \geq 1$. Furthermore,  from the definition \eqref{dePhik} of $\Phi_0$ and  the formula \eqref{eQ.index} representing the index, 
we see that 
\begin{align*}
 \lim_{\varepsilon \to 0}  \left( - \int_{\partial B_i(\varepsilon)} \Phi_0 \right) 
&=  \lim_{\varepsilon \to 0}  \frac{1}{(m-1)!}\cdot  
\mathcal{Q}_0(h \mid {\partial B_\varepsilon},g)
=  (m-1) !  \Vol (S^{m-1}) \Ind (\xi, x_i).
\end{align*}
We have thus established  that
\begin{align*}
 \int_{M} \Pf(\Omega) - \int_{\partial M} \Phi   &= 
 \lim_{\varepsilon \to 0} \,  \sum_{i=1}^q  \left( -a_{m,0}\right) \int_{\partial B_i(\varepsilon)} \Phi_0 
 \\ & = 
\lim_{\varepsilon \to 0} \,  \sum_{i=1}^q \frac{a_{m,0}}{(m-1)!}\cdot  
\mathcal{Q}_0(h \mid {\partial B_\varepsilon},g)
\\&=  a_{m,0} (m-1) !  \Vol (S^{m-1})  \sum_{i=1}^q\Ind (\xi, x_i).
\end{align*}

Recall that $a_{m,0} = \dfrac{1}{(m-1)!!}$. Furthermore,  since  $m$ is even, we have $\Vol(S^{m-1}) = \dfrac{(2\pi)^{{\frac{m}{2}}}}{(m-2)!!}$, which leads to the following relationship:
\[
a_{m,0} (m-1)! \, \Vol(S^{m-1}) = (2\pi)^{\frac{m}{2}}.
\]
Applying the Poincaré--Hopf Theorem, we finally conclude that  
\begin{align*}
   \int_{M} \Pf(\Omega) - \int_{\partial M} \Phi    =    (2\pi)^{\frac{m}{2}}   \sum_{i=1}^{q} \Ind(\xi)  =  (2\pi)^{\frac{m}{2}}  \chi(M).
\end{align*}
\end{proof}

\medskip

Let us pause and take a short moment to listen to Chern's own account of his discovery,  as he recalls it in an interview with Allyn Jackson in 1998 (see \cite{Interview}). 

\smallskip  

\begin{quotation}
\fontfamily{libertine}\selectfont 
\begin{minipage}{0.8\textwidth} 
{\small 
The Gauss--Bonnet formula is one of the important, fundamental formulas, not only
in differential geometry, but in the whole of mathematics. Before I came to Princeton [in 1943] I had thought about it, so the development in Princeton was in a sense very natural. I came to Princeton and I met André Weil. He had just published his paper with Allendoerfer. Weil and I became good friends, so we naturally discussed the Gauss--Bonnet formula. And then I got my proof. I think this is one of my best works, because it solved an important, a fundamental, classical problem, and the ideas were very new. And to carry out the ideas you need some technical ingenuity. It’s not trivial. It’s not something where once you have the ideas you can carry it out. It is subtle.  So I think this a very good piece of work.
}
\end{minipage}
\end{quotation}

\subsection{Demystifying the Gauss--Bonnet Boundary term using Double Forms}
\label{sec.Chern-Gauss--Bonnet_Bis}

In this section, we use the formalism of double forms to introduce two alternative formulations of the Gauss--Bonnet Formula for a compact, oriented Riemannian manifold $(M, g)$ with  non-empty boundary and  even dimension $m$.
The first reformulation involves  the curvature tensor $R$ of the oriented Riemannian manifold $(M, g)$, coupled with the second fundamental form $h$ of its boundary. The formula can be expressed as follows:
 
\begin{equation}\label{GaussBonnetQB}
   (2\pi)^{\frac{m}{2}}  \chi(M) =  \mathcal{K}_m(M) +
  \sum_{k=0}^{\frac{m}{2}-1} b_{m,k}\,   \mathcal{Q}_k(R,h \mid \partial M, \bar{g}).
\end{equation} 
Here  $\bar{g}$  is the first  fundamental form of $\partial M$ as a hypersurface in  $(M,g)$.
We will  prove this formula  and determine the coefficients $b_{m,k}$ below. 

\medskip

We work with the orientation of  the boundary $\partial M$  induced by the unit outer normal field $\nu \in T(\partial M)^{\perp}$.  This means that a basis $v_1, \ldots, v_{m-1}$ of $T_p(\partial M)$
is positively oriented if and only if $\nu, v_1, \ldots, v_{m-1}$ is a positively oriented basis in $T_p M$. 
Referring to \eqref{normalvarpi}, we see that,  under this orientation, the volume form on the boundary  is given by
$$
  \dvol_{\partial M} = \varpi_{\nu}.
$$
We next choose  a vector field $\zeta$ of length $1$, defined in some neighborhood of $\partial M \subset M$ and such that 
$\zeta$  is outward pointing  and orthogonal to $\partial M$.  Observe that $\zeta= \nu$ at every boundary point $p\in \partial M$. 
We then define $\Phi_k$ and $h$ as previously, that is,  
\begin{align*}
  h(X,Y)  =  \bra X, \nabla_Y \zeta\ket
     \quad \text{and } \quad 
 \Phi_k = - \frac{2^k}{(m-1)!} \cdot \Ct^{m-1} \left( R^k \owedge h^{m-1-2k} \right) \varpi_\zeta,
\end{align*}
and observe that the restriction of $h$ to the boundary coincides with  its second fundamental form.\footnote{In the literature, the second fundamental form is frequently defined using the opposite sign convention. The present  convention yields positive principal curvatures for the boundary of the unit ball in Euclidean space.}

\smallskip 
 
From Definition \ref{defQ},   we then have
\begin{align*}
   \mathcal{Q}_k(R,h \mid \partial M) &=  \frac{1}{(m-1)!\,  k ! \, (m-2k-1)!} \int_{\partial M} \Ct^{m-1}({R}^k \owedge h^{m-2k-1}) \dvol_{\partial M}
  \\&=    \frac{-1}{2^k\,  k ! \, (m-2k-1)!} \int_{\partial M} \Phi_k.
\end{align*}
So if we now define the coefficients $b_{m,k}$ by the following relation: 
$$
 b_{m,k}\,   \mathcal{Q}_k(R,h \mid \partial M) = - a_{m,k}\, \int_{\partial M} \Phi_k,
$$
we   conclude that 
\begin{align*}
  \mathcal{K}_m(M) +   \sum_{k=0}^{\frac{m}{2}-1} b_{m,k}\,   \mathcal{Q}_k(R,h \mid \partial M, \bar{g})
   =  \int_M \Pf(\Omega) \dvol_M - \int_{\partial M} \Phi =  (2\pi)^{\frac{m}{2}}  \chi(M). 
\end{align*}
Formula  \eqref{GaussBonnetQB} is proved. 
Furthermore, the coefficients are calculated as follows:
\begin{equation}\label{Calculbb}
    {b_{m,k}}  =   a_{m,k} \cdot {2}^{k} \,  k! \, \left( m-2\,k-1 \right)!  = \frac{\left( m-2\,k-1 \right)!}{\left( m-2\,k-1 \right)!!} 
    =    (m-2 k-2)!!
\end{equation}

\medskip

The second reformulation of the Gauss--Bonnet Formula expresses the boundary term using the intrinsic curvature tensor $\bar{R}$ of the boundary hypersurface

\begin{equation}\label{GaussBonnetQC}
   \mathcal{K}_m(M) +
  \sum_{k=0}^{\frac{m}{2}-1} c_{m,k}\,   \mathcal{Q}_k(\bar{R},h \mid \partial M, \bar{g})
    =     (2\pi)^{\frac{m}{2}}  \chi(M),
\end{equation} 
where, again, the coefficients $c_{m,k}$ are to be determined. 
Using  the Gauss Equation 
$$R = \bar{R} - \frac{1}{2} h\owedge h, $$
 together with  the linearity of the contraction, the fact that $\bar{R}$ and $h \owedge h$ commute and  the binomial formula, we see that the following relation holds on the boundary:
$$
 R^p  \owedge h^{m-1-2p} = \sum_{k=0}^p \binom{p}{k} \left( -\tfrac{1}{2}\right)^{p-k} 
\bar{R}^k  \owedge h^{m-1-2k},
$$
therefore
$$
 \Ct^{m-1}\left( R^p  \owedge h^{m-1-2p}\right)  = 
\sum_{k=0}^p \binom{p}{k}   \left( -\tfrac{1}{2}\right)^{p-k}  \,  \Ct^{m-1}\left(\bar{R}^k  \owedge h^{m-1-2k}\right).
$$
From \eqref{eq.defQ}   and the above relation  we obtain: 
$$
 \mathcal{Q}_p\left(R,h \mid\partial M, \bar{g}\right) 
=  \sum_{k=0}^p   w_{m,p,k}  \mathcal{Q}_k(\bar{R},h \mid\partial M, \bar{g}),
$$
where
$$
   w_{m,p,k}   =    \left( -\tfrac{1}{2}\right)^{p-k}  \frac{(m-1-2k)!}{(m-1-2p)! \, (p-k)!}  
$$
if $0 \leq k \leq p$. It will be convenient to also define  $ w_{m,p,k}=0$ if $k>p$. We then have  
$$
  \sum_{p=0}^{\frac{m}{2}-1} b_{m,p}  \mathcal{Q}_p\left(R,h \mid\partial M, \bar{g}\right) 
 =  \sum_{k=0}^{\frac{m}{2}-1} c_{m,k} \mathcal{Q}_p\left(\bar{R},h \mid\partial M, \bar{g}\right)),
$$
where $c_{m,k}$ is defined as 
$$
  c_{m,k}  = \sum_{p=k}^{\frac{m}{2}-1}  w_{m,p,k}b_{m,p}.
$$
We have thus proved that \eqref{GaussBonnetQB} holds with 
\begin{align}\label{eq.cmk}
 c_{m,k}  &=  (m-2k-1)!  \sum_{p=k}^{\frac{m}{2}-1}  (-1)^{p-k}   
 \frac{( m-2p-2)!! }{ 2^{p-k}(m-2p-1)! \, (p-k)!}  \nonumber 
\\&=   (-1)^{\frac{m}{2}  -k -1} (m-2k-3)!!
\end{align}

Synthesizing our findings, we arrive at  the following formulations of the Gauss--Bonnet--Chern Formula:
\index{Gauss--Bonnet--Chern Formula}

\begin{equation}\label{NewGBC}
\boxed{
\begin{aligned}
  (2\pi)^{\frac{m}{2}}  \chi(M) &=  \int_M \Pf(\Omega) -  \sum_{k=0}^{\frac{m}{2}-1} a_{m,k} \int_{\partial M} \Phi_k 
 \\
   &=  \mathcal{K}_m(M) + \sum_{k=0}^{\frac{m}{2}-1} b_{m,k}\,   \mathcal{Q}_k(R,h \mid \partial M, \bar{g}) \\
   &= \mathcal{K}_m(M) +
 \sum_{k=0}^{\frac{m}{2}-1} c_{m,k}\,   \mathcal{Q}_k(\bar{R},h \mid \partial M, \bar{g}),
\end{aligned}
}
\end{equation}

where the coefficients are
\begin{align*}
    a_{m,k} &=  \frac{ 1 }{ (2k)!!\, \left( m-2\,k-1 \right) !!}\\
    b_{m,k} &= (m-2k-2)!! \\
    c_{m,k} &=  (-1)^{\frac{m}{2}-k-1}  (m-2k-3)!!
\end{align*}
We also remind the reader that 
$$
   \int_M \Pf(\Omega) =  \int_M \RR \dvol_M =  \mathcal{K}_m(M).
$$

\medskip 

\begin{remark}\rm 
Perhaps one  ought to  assign a name to the tensors  $R^k \owedge h^{m-1-2k}$ appearing in the preceding formulas. 
Although the term \textit{Generalized Gauss--Kronecker Curvature} comes to mind, this terminology has already been adopted in \cite{Kulkarni} for the tensor $R^k$. An alternative suggestion could be the  \textit{Gauss--Kronecker--Chern Curvatures}.

\smallskip 

It is noteworthy that the boundary term in the Gauss--Bonnet Formula appears as the integral of a weighted homogeneous polynomial of degree $m-1$ in $R$ (the curvature tensor) and $h$ (the second fundamental form), with weights of $1$ assigned to $h$ and $2$  assigned to $R$.

\smallskip 

Another point that might be worth the effort would be to try and find a more intuitive explanation for the values of the coefficients $b_{m,k}$ and $c_{m,k}$ (in particular the presence of double factorials here is no coincidence).
\end{remark}

\section{Examples and Applications}

\subsection{Flat Manifolds.} \label{sec.flat}

Recall that a Riemannian manifold $(M,g)$ is \textit{flat} if its curvature tensor satisfies $R = 0$. Equivalently, every point $x\not\in \partial M$ admits a neighborhood isometric to an open subset of $\R^n$. 

\begin{theorem}\label{th.flatcase}
If $(M,g)$ is a compact flat oriented  Riemannnian manifold with boundary of dimension $m$, then 
$$
  \int_{\partial M} K \dvol_{\partial M} = \Vol(S^{m-1}) \, \chi(M),
$$
where $K = \det(h_{ij})$ is the Gauss--Kronecker curvature of the boundary $\partial M$.
\end{theorem}

\begin{proof} The proof splits in two cases according to whether $m$ is even or odd. 
Consider first the case where $m$ is even,. We have $\mathcal{K}_m(M) = 0$ since $R = 0$. We also have    
$\mathcal{Q}_k(R,h \mid \partial M, \bar{g}) = 0$ for any $k\geq 1$. Because  $b_{m,0} = (m-2)!!$,   the second formula in \eqref{NewGBC} reduces thus to
$$
 \int_{\partial M} K \dvol_{\partial M} =  \mathcal{Q}_0(h \mid \partial M) = \frac{(2\pi)^{{\frac{m}{2}}}}{(m-2)!!} \,  \chi(M)
=  \Vol(S^{m-1}) \chi(M).
$$ 
If $m$ is odd, then $N = \partial M$ is an even-dimensional manifold and we know from  Equation \eqref{GBHopf2} that 
$$
 \int_{N} K \dvol_{N}  = \frac{1}{2} \Vol(S^{m-1})\, \chi(N)
$$
(see also Proposition \ref{GK_intrinseque}). 
We  conclude the proof with  the relation $\chi(\partial M) = 2 \chi(M)$, which holds when $m = \dim (M)$ is odd.
\end{proof}

\medskip
 
An interesting outcome of this theorem is the following generalization of Corollary \ref{cor.normaldegree}.
Suppose that an immersion $f : M \to \R^m$ is given, where 
$M$ is a smooth compact manifold of (even or odd) dimension $m$. We can then prescribe an orientation on $M$ by requiring $f$ to be a positively oriented map and then define the Gauss map $\nu : \partial M \to S^{m-1}$. 
\begin{corollary}\label{cor.normaldegree2}
The degree of the Gauss map $\nu   : \partial M \to S^{m-1}$ is equal to the Euler characteristic of $M$:
$$
 \deg(\nu) = \chi(M).
$$
\end{corollary}
\begin{proof}
Observe first that the pullback $g$ of the Euclidean metric by the map  $f : M \to \R^m$ is a flat metric on $M$.  
The proof is now obvious from  Theorem \ref{th.flatcase} since the Gauss--Kronecker curvature is the Jacobian of the Gauss map.
\end{proof}

This result  was   independently discovered in 1960 by  A. Haefliger and H. Samelson, see  \cite[Theorem 3bis]{Haefliger1960} and  \cite[Theorem 1]{Samelson}. In  the special case where $M$ is a domain in $\R^n$, H. Hopf had already established it  in \cite[Satz VI.]{Hopf1926}. Some historical context  is discussed in  Gotlieb's essay \cite{Gottlieb1996}.

\subsection{Manifolds with Flat Boundary}

Consider now the case of a compact oriented  Riemannian manifold $(M,g)$, of even dimension $m$, whose boundary is flat, that is, $\bar{R} = 0$. The second formula in \eqref{NewGBC} tells us that 
\begin{equation}\label{GBFlatbdy}
 \int_M \RR \, \dvol_M   +  \int_{\partial M} K \dvol_{\partial M}  = (2\pi)^{{\frac{m}{2}}}\chi(M),
\end{equation}
where $\RR$ is the Gauss--Bonnet density of $(M,g)$ and  $K = \det(h_{ij})$ is the Gauss--Kronecker curvature of the boundary $\partial M$. 

When $m=2$, this  recovers  the classical  Gauss--Bonnet Formula for compact surfaces with smooth boundary (up to a change in notation). In this case, we have indeed  $\bar{R} = 0$ because  the boundary is  a one-dimensional manifold.

\subsection{Space Forms} \label{sec.spaceform}

The Gauss--Bonnet density of a manifold  $(M,g)$ of constant sectional curvature $a$ and even dimension $m$ is a constant, equal to 
\begin{equation}\label{GBSpaceForm}
 \RR = \left(\tfrac{a}{2}\right)^{\frac{m}{2}} \frac{m!}{\left(\frac{m}{2} \right)! }
 =
   \frac{ (2\pi a)^{\frac{m}{2}}}{\frac{1}{2}\Vol(S^m)}.
\end{equation}

If $M$ is compact,  then the Gauss--Bonnet Formula immediately  implies the following identity, which was already proved in 1925 by Hopf in \cite{Hopf1925a}:
\begin{equation}\label{spaceform}
 a^{\frac{m}{2}}   \Vol(M) = \frac{1}{2}\Vol(S^m) \chi(M).
\end{equation}

In particular, when $a>0$, the Euler characteristic $\chi(M)$ is  positive, it vanishes in the flat case ($a=0$), and its sign is $(-1)^{\frac{m}{2}}$ when $a<0$. 
Note  that Equation \eqref{spaceform} holds trivially for flat manifolds, while it becomes almost self-evident for spherical manifolds due to the fact that the universal cover of such a manifold is a sphere.
A geometric proof, based on  a geodesic triangulation, is presented in \cite[Theorem 11.3.2]{Ratcliffe2019}.
Let us add a few comments on this result:

\begin{enumerate}[(1)]
\item The formula \eqref{GBSpaceForm} also works for the case of \textit{non compact, complete hyperbolic manifolds of finite volume} of even dimension. This was first proved by  R. Kellerhals and T. Zehrt \cite{Kellerhals1984} using an ideal geodesic triangulation of the manifold. It can also be seen as a consequence of \eqref{GBFlatbdy}.
Here is a brief  explanation of why this holds: such a manifold is the union of a compact manifold with horospherical boundary and a finite collection of cusps. Each connected component of the boundary is a compact flat manifold with constant Gauss--Kronecker curvature $K = +1$ (in fact we have  $h = \bar{g}$, equivalently all the principal curvature are equal to $+1$). Progressively chopping off the manifold at increasing distances, it becomes evident that $\int_{\partial M} K \dvol_{\partial M} \to 0$ and we conclude by \eqref{GBFlatbdy}.
\item  Our considerations prove that the  volume of a complete  hyperbolic manifold of even dimension $m$ must be an  integer multiple of $\frac{1}{2}\Vol(S^m)$. Furthermore, if the manifold has dimension $4$, we know from the work of  J. Ratcliffe and T. Tschantz  \cite{RatcliffeTschantz} that every positive multiple of $\frac{4}{3}\pi^2$ is realized as the volume of a complete 4-dimensional hyperbolic manifold.
\item The situation is very different for odd-dimensional hyperbolic manifolds. A fundamental result by  W. Thurston and  O. Jørgensen,    based on a construction involving hyperbolic Dehn surgeries, shows that the set of volumes of complete hyperbolic 3-manifolds is a closed and non-discrete subset of $\R$, containing infinitely many accumulation points,
see  \cite{Thurston} and  \cite{Gromov1981}.
\end{enumerate}

\subsection{Further results on non Compact Manifolds}

An almost  immediate consequence of \eqref{GaussBonnetQC} is  the following 

\begin{theorem}
Let $(M, g)$ be a non-compact Riemannian manifold of even dimension $m$ and finite topological type. Assume $g$ has an integrable Gauss--Bonnet density, and that  there exists an exhausting sequence
$$
M_1 \subset \cdots \subset M_{j} \subset M_{j+1} \subset M = \bigcup_{j=1}^{\infty} M_{j},
$$
such that each $M_{j}$ is a compact $m$-dimensional manifold with smooth boundary. If
\begin{equation}\label{cond.exhaustion}
\int_{\partial M_{j}} \left( \|\bar{R}_x\|^k\|h_x\|^{m-1-2k}\right) \dvol_g \to 0
\quad \text{as } j \to \infty
\end{equation}
for any integer $k$ such that $0 \leq k \leq  \frac{m}{2}-1$, then the Gauss--Bonnet Formula holds; that is, we have
\begin{equation}\label{gbkm}
(2\pi)^{\frac{m}{2}} \chi(M) = \mathcal{K}_m(M, g).
\end{equation}
\end{theorem}

\smallskip

By an observation of M. Gromov, this result applies  to the case of complete Riemannian manifolds with pinched negative sectional curvature, that is, whose sectional curvature satisfies $-a \leq K \leq -b < 0$, see \cite[appendix 3]{Gromov1982}.  When the manifold has non-positive curvature, the control of the geometry at infinity is more delicate but some explicit asymptotic conditions have been given by  S. Rosenberg in \cite{Rosenberg}; see also \cite{Harder} for the case of arithmetic manifolds. The paper \cite{CheegerGromov} by Cheeger and Gromov contains a strikingly stronger result. \medskip

\subsection{The direct product of a Ball and a Sphere}
\label{ExampleBS}

Our next example will illustrate that, even in a very simple case,  keeping track of   the boundary terms in the Gauss--Bonnet--Chern Formula can be tedious. Nevertheless, this example holds significance as it validates the accurate values of the coefficients $b_{m,k}$ in Formula \eqref{NewGBC}.

\smallskip 

Let us denote the standard metric on $\R^{p+1}$ as $g_1$ and the standard metric on $\R^{q+1}$ as $g_2$. Their respective restrictions to the unit spheres $S^p$ and $S^q$ are denoted by $\bar{g}_1$ and $\bar{g}_2$. The Riemannian metric on the manifold $M$ is given by $g = g_1 + \bar{g}_2$, and its restriction to $\partial M$ is $\bar{g} = \bar{g}_1 + \bar{g}_2$. It is  important to note that, in the algebra of double forms, $\bar{g}_1^j = 0$ if $j > p$ and $\bar{g}_2^j = 0$ if $j > q$.

\medskip 

With this notation, one  can write the curvature tensor $R$ of $M$ as  $R = \frac{1}{2} \bar{g}_2 \owedge \bar{g}_2$ (because $B^{p+1}$ is  flat, all the curvature comes from $S^q$).
It follows immediately that the Gauss--Bonnet density defined in \eqref{eq.defRR}  vanishes on $M$:
\begin{equation*} 
  \RR = \tfrac {1}{m!\left(\frac{m}{2}\right)! } \, \Ct^m(R^{ {\frac{m}{2}}}) =  \tfrac {1}{2^m m!\left(\frac{m}{2}\right)! } \,  \Ct^m(\bar{g}_2^m) = 0,
\end{equation*}
because $m > q$. 
Furthermore,  the second fundamental form of the boundary $\partial M = S^p\times S^q$ is given by $h = \bar{g}_2$.  
We therefore have for any $0\leq k \leq \frac{1}{2}(p+q)$ 
$$
  R^k\owedge  h^{m-1-2k} = \left(\tfrac{1}{2}\right)^{k} \bar{g}_1^{2k} \owedge  \bar{g}_2^{m-1-2k},
$$
which vanishes if either $2k > p$ or $(m-1-2k) > q$. This implies that 
$$
  R^k\owedge  h^{m-1-2k} = \begin{cases}   \left(\tfrac{1}{2}\right)^{\frac{p}{2}} \bar{g}_1^{p} \owedge  \bar{g}_2^{q}, & \text{ if } p=2k,
 \vspace{0.14cm} \\
  \; 0  & \text{ else. } \end{cases}
$$
Note also that 
$$
  \bar{g}^{m-1} =  \bar{g}^{p+q}  =  \sum_{j=0}^{p+q} \binom{p+q}{j}  \,  \bar{g}_1^j \, \owedge \bar{g}_2^{p+q-j}
  = \frac{(p+q)!}{p! \, q!}   \bar{g}_1^p \,  \owedge \bar{g}_2^{q}.
$$
We thus have  $\mathcal{Q}_k(R, h \mid \partial M, \bar{g}) =0$ if $k\neq \frac{p}{2}$ and 
\begin{align*}
 \mathcal{Q}_{\frac{p}{2}} (R, h \mid \partial M, \bar{g}) 
&=  \frac{1}{(p+q) ! \, \left(\frac{p}{2}\right)  ! \, q! }  \, 
 \int_{\partial M}  \Ct^{p+q}\left( \left(\tfrac{1}{2}\right)^{\frac{p}{2}}\bar{g}_1^{p} \owedge  \bar{g}_2^{q} \right)  \dvol_{\partial M} 
\\ &
= \frac{1}{(p+q) ! \, \left(\frac{p}{2}\right)  ! \, q! }  \, 
 \int_{\partial M}  \Ct^{p+q}\left( \frac{p! \, q!}{2^{\frac{p}{2}} (p+q)!} \,   g^{p+q}  \right)  \dvol_{\partial M} 
\\&=
   \frac{p!}{ 2^{\frac{p}{2}}   \left(\frac{p}{2}\right)  ! }  \, \Vol(S^p) \Vol(S^q).
\end{align*}
The latter quantity can be explicitly computed. If  $p$ is even and $q$ odd, we have 
\begin{equation}\label{SpSq}
\Vol(S^p) \Vol(S^q)
 =  
    \frac{ 2^{\frac{p}{2}+1}   \left(\frac{p}{2}\right)  ! }{p! \, (q-1) !! }  \cdot  (2\pi)^{\frac{p+q+1}{2}}.
\end{equation}
We finally conclude that 
\begin{align*}
   \mathcal{K}_m(M) &  + \sum_{k=0}^{\frac{m}{2}-1} b_{m,k}\,   \mathcal{Q}_k(R,h \mid \partial M, \bar{g})   
\\& = 0 + 
 b_{p+q+1,\frac{p}{2}} \cdot \mathcal{Q}_{\frac{p}{2}} (R, h \mid \partial M, \bar{g}) 
\\ &=
(q-1) !!  \  \frac{p!}{ 2^{\frac{p}{2}}   \left(\frac{p}{2}\right)  ! }  \, \Vol(S^p) \Vol(S^q)
\\ &= 
 2 \, (2\pi)^{\frac{p+q+1}{2}}  
\\ &= 
  (2\pi)^{\frac{m}{2}}  \chi(M).
\end{align*}

\medskip 

\begin{remark} \rm 
We emphasize that this calculation is independent of the proof of Proposition \ref{Primitive}. Consequently, it offers an alternative approach to determining the exact value of the coefficients $b_{m,k}$. 
This example is inspired by the closing sentence in Volume 5 of Spivak's treatise (page 391). Spivak suggests working out the above computation and adds that  ``after performing the calculation, it should be fun to compare with Chern's paper''.
\end{remark}

\subsection{Rotationally Symmetric Metrics}
\label{rotmetric}  \index{Rotationally Symmetric Metrics}
 
In our last example, we  consider the case of a rotationally symmetric smooth Riemannian metric $g$ in a closed  ball  $B_r$  of radius $r$. 
Denoting  by $o\in B_m(r)$ its center, we may identify   ${B}^m\setminus \{o\}$ with $(0,r] \times S^{m-1}$, and the metric tensor can be represented as a warped product 
$$
 g = dt^2 + f(t)^2 \, g_0,
$$
where $g_0$ is the standard metric on the unit sphere $S^{m-1}$.
The second fundamental form and the   intrinsic curvature tensor of the boundary sphere $S_r = \partial B_r$  are then given by 
$$
  h = \frac{f'(r)}{f(r)}\bar{g}  = f'(r) f(r) \,  g_0  \quad \text{and} \quad 
 \bar{R} = \frac{g \owedge g}{2\, f(r)^2}  = \frac{1}{2} \, g_0 \owedge g_0.
$$
The boundary contribution to the Gauss--Bonnet Formula \eqref{GaussBonnetQB} appears as a sum of terms of the following type:
$$
    c_{m,k}\, \mathcal{Q}_k(\bar{R}, h \mid S_r, \bar{g})  = \gamma_{m,k} \Vol\left(S^{m-1}\right)  \,  (f'(r)) ^{m-2k-1}, 
$$
for some coefficient $\gamma_{m,k}$. The courageous reader will compute and find the value 
$$
 \gamma_{m,k} =  {\frac { (m-1) ! \,  c_{m,k}}{{2}^{k}k!\,  \left( m-1-2\,k \right) !}}   
 =  \frac { (-1)^{\frac{m}{2}-k-1} \left( m-1 \right) ! }{{2}^{\frac{m}{2}-k}  k!\,\left( m-2k-1 \right)  \left( \frac{m}{2}-k-1 \right) !}.
$$
The  Gauss--Bonnet Formula for $(B_r,g)$ is then the following identity:
\begin{equation}\label{GBspherical}
 \int_{B_r} \RR \dvol_g +  \Vol(S^{m-1}) \sum_{k=0}^{\frac{m}{2}-1}  \gamma_{m,k} \,  (f'(r)) ^{m-2k-1} = (2\pi)^m.
\end{equation}

Let us look more closely at  the cases  of  the  Euclidean, spherical  and  hyperbolic balls. 
In the Euclidean case, we have  $\RR = 0$ and $f(t) = t$. Then $f'(t) = 1$  and \eqref{GBspherical} reduced to the following summation of the coefficients $ \gamma_{m,k}$, which holds when $m$ is an even integer
\begin{equation}\label{sumgamma}
 \sum_{k=0}^{\frac{m}{2}-1}   \gamma_{m,k}   \ = \frac{(2\pi)^{\frac{m}{2}}  }{\Vol(S^{m-1})} =    
(m-2)!!  
\end{equation}
Consider next the ball of radius $r$ in the hyperbolic space $B_r \subset \mathbb{H}^m$ (with $m$ even). We have $f(r) = \sinh(r)$, and therefore \eqref{GBspherical} gives us 
\[
 \int_{B_r} \RR \dvol_g  =  \frac{(-2\pi)^{\frac{m}{2}}}{\frac{1}{2} \Vol(S^m)}
 \Vol_{_{\mathbb{H}^m}}(B_r) = 
 \frac{(-2\pi)^{\frac{m}{2}}\Vol (S^{m-1})}{\frac{1}{2} \Vol(S^m)} \int_0^r \left(\sinh(t)\right)^{m-1} dt.
\]
Using now \eqref{GBSpaceForm}, we see that the Gauss--Bonnet  Formula  \eqref{GBspherical} for the hyperbolic ball  is equivalent to the identity
\[
{\frac { (-1)) ^{\frac{m}{2}}  m!}{{2}^{\frac{m}{2}}  \left( \frac{m}{2} \right) !}}
\,  \int_{0}^{r}\!  \left( \sinh \left( t \right)  \right) ^{m-1}\,{\rm d}t
 \  + \ 
 \sum _{k=0}^{\frac{m}{2}-1}  \gamma_{m,k}
 \left( \cosh \left( r \right)  \right) ^{m-2\,k-1}
 \  = \ 
(m-2)!! 
\]
A direct  derivation of this identity (without invoking the Gauss--Bonnet Formula) can be achieved through an induction-based argument, starting from the relation:
$$
 (m-1) \int_0^r (\sinh(t))^{m-1}(t) dt = \cosh(r)\left(\cosh(r)^2-1\right)^{\frac{m}{2}-1} - (m-2)  \int_0^r (\sinh(t))^{m-3}(t) dt,
$$
which  is easily verified with an integration by parts.

\medskip 

The argument and the calculations are similar in the spherical case. The corresponding integral identity being
$$
{\frac {m!}{{2}^{\frac{m}{2}} \left( \frac{m}{2} \right) !}}
\,  \int_{0}^{r}\!  \left( \sin \left( t \right)  \right) ^{m-1}\,{\rm d}t
 \  + \ 
 \sum _{k=0}^{\frac{m}{2}-1}  \gamma_{m,k}
 \left( \cos \left( r \right)  \right) ^{m-2\,k-1}
 \  = \ 
(m-2)!! 
$$
We leave the detailed verification  to the interested reader. 

\section{Miscellaneous Remarks} 

\subsection{Remarks on Signs and Other Conventions}

In the literature, including in Spivak's books, the curvature tensor $R$  is often defined with the opposite sign. Nonetheless, we chose to follow the convention as in Gray's book \cite{Gray}, as it aligns well with the formalism of double forms.
However, the crucial point to note is that the signs in our Lemma \ref{lem:formule_R} match those in \cite{Spivak} and \cite{Willmore}.

\medskip

These references have also a different convention on  indices. They write  the connection and curvatures forms as
\begin{equation}\label{upindx}
 \omega^i_j = \omega_{ij} \quad \text{and} \quad \Omega^i_j = \Omega_{ij}.
\end{equation}
Our convention is justified by our choice to work with the covariant curvature tensor $R_{ijkl}$ rather than the mixed one $R^i_{\: ijkl}$. Observe however that  both index conventions yield the same quantities, as the general transformation rule
$$
 \Omega^i_j = \sum_k g_{ik}\, \Omega^k_{j} 
$$
reduces to \eqref{upindx} due to the orthonormality of the moving frame.

\subsection{Comparison with  Chern's Papers}  \label{sec.compchern}

We strongly encourage the reader to look into Chern's two original articles. Although not easily accessible, they are concise, elegant, and remarkably insightful. Both papers start with a succinct review of the moving frame method, in Élie Cartan's style.
In this approach, given a moving frame $Pe_1 \cdots e_m$, the following equations are posed:
\begin{eqnarray*}
    &dP = \sum_i \omega_i e_i, 
 \hspace{2cm} 
  &\Omega_{ij} = d\omega_{ij} - \sum_k \omega_{ik}\omega_{kj}, 
\\
   &de_i = \sum_j \omega_{ij}e_j,
 \hspace{2cm} 
    &d\Omega_{ij} = \omega_{ik}\Omega_{kj} - \omega_{jk}\Omega_{ki}.
\end{eqnarray*}

If, following Cartan, we interpret the symbol $P$ (the ``point $P$ itself'') as representing the identity map of the manifold to itself, then $dP$ is the identity map in the tangent space. The first equation tells us that $\omega_i = \theta^i$ is the dual coframe.
The second and third equations are the structure equations, with $de_i$ interpreted as $\nabla e_i$. The last equation represents the Bianchi identity.  Cartan, and  Chern after him, use these equations as \textit{definitions} of the connection and curvature forms. Note  however that these definitions yield the   opposite sign compared to what appears in the present paper:
$$
\omega_{ij}^{\text{Chern}} =  -\omega_{ij}, \qquad \Omega_{ij}^{\text{Chern}} =  -\Omega_{ij}.
$$
This differences in sign are reflected in several equations.  Another important difference is that Chern does not state our Proposition \ref{Primitive} concerning manifolds with a vector field. Instead, he works with the unit tangent bundle $SM$ of the manifold $M$ and proves that the pullback of the Pfaffian $\Pf(\Omega)$ is an exact form on $SM$. This operation is called a  \textit{transgression} of the corresponding differential form. The calculation is essentially the same as the one presented in our  proof of  Proposition \ref{Primitive}.

\smallskip 

Our presentation of the proof shows that working on the sphere bundle and transgressing the Pfaffian are not required    to work out Chern's argument and prove the intrinsic Gauss--Bonnet Formula. However, the concept  of transgression has played a central  role   in the theory of characteristic classes.

\subsection{Double Factorials and the volume of Spheres}
\label{sec.volumeSd}

Here we briefly discuss some  quantities that  have been used throughout  the paper.  \index{Double factorial}
The \textit{double factorial} is inductively defined for any integer $n \geq -1$ by the conditions $ (-1)!! = 0!! = 1$ and 
$$
  n!! = n(n-2)!!,
$$
for $n\geq 1$. We thus have for even integers:
$$
   n!! = n \cdot(n-2) \cdot(n-4) \cdots  2 = 2^{\frac{n}{2}} \cdot\left(\tfrac{n}{2}\right)!
$$
and for odd integers 
$$
   n!!   = n \cdot(n-2) \cdot(n-4) \cdot1 =  
\frac{(n+1)!}{2^{\frac{n+1}{2}} \cdot\left(\frac{n+1}{2}\right)!} = 
 \frac{2^{\frac{n+1}{2}} \,\Gamma\left(\tfrac{n+1}{2}\right) }{\sqrt{\pi}}, 
$$
where $\Gamma(z) = \int_{0}^{\infty} t^{z-1} e^{-t} \, dt$ \  is the Euler Gamma function. 

\medskip

The double factorial has the following simple combinatorial interpretation:
For  a set with $n$ elements,  the number of different  ways to form $\tfrac{n}{2}$  unordered pairs  if $n$ is even, or  $\tfrac{n-1}{2}$  unordered pairs with one element left unmatched if $n$ is odd, is given by
$$
 \begin{cases}
 (n-1) !! & \text{ if  $n$ is even,}  \\   \quad n!! & \text{ if  $n$ is odd.} 
\end{cases}
$$

\medskip

The other ubiquitous quantity is the volume (or should we call it the area?) of the unit sphere $S^d \subset \R^{d+1}$. It  can be written in several different formalisms, starting with 
$$
  \Vol(S^d) = 2\,{\frac {{\pi}^{\frac{d+1}{2}}}{\Gamma \left( \frac{d+1}{2} \right) }}.
$$
We see that for $d$ odd we have
$$
  \Vol(S^d) = 2\,{\frac {{\pi}^{\frac{d+1}{2}} }{ \left(\frac{d-1}{2}\right) !}}
= {\frac { \left( 2\,\pi \right) ^{\frac{d+1}{2}}}{ \left( d-1 \right) !!}},
$$
and for even $d$:
$$
  \Vol(S^d) = {\frac{{2}^{d+1}{\pi}^{\frac{d}{2}} \left( \frac{d}{2} \right) !}{d!}}
 = {\frac { 2 \left( 2\,\pi \right)^\frac{d}{2}  }{ \left( d-1 \right) !!}}.
$$

\section{Exercises}
\label{sec.exercices}

In conclusion of  this paper, we propose  a few exercises.   

\bigskip 
\textbf{1.} 
Verify that, in the case of a surface, each of the three formulations  \eqref{NewGBC} yields the classical  2-dimensional Gauss--Bonnet Formula  (this amounts essentially to identifying the boundary term as the integral of the geodesic curvature,  with the correct orientation).

\bigskip

\textbf{2.}  
\index{Pfaffian}
Recall that the Pfaffian of an \(m\times m\) matrix \(A\), where  $m$ is even, is the following sum involving $m!$ terms:
\[
 \Pf(A) = \frac{1}{m!!} \sum_{\sigma \in S_{m}} (-1)^{\sigma} a_{\sigma(1)\sigma(2)} \cdot a_{\sigma(3)\sigma(4)} \cdots  a_{\sigma(m-1)\sigma(m)}.
\]
In this excercice we will show that the  Pfaffian can be rewritten as a sum of only $(m-1) !!$ \  terms, each involving a different partition of the index set $\{1, \dots, m\}$ in disjoint pairs. To this aim, we consider the following subset 
$$
 \mathcal{P}_m = \{ \sigma \in S_m \mid \sigma_{2i+2} > \max\{ \sigma_{2i}, \sigma_{2i+1}\} , \;  1 \leq i <  \tfrac{m}{2}\}  
$$
of $S_m$.   Your task is to first verify  the following identity: 
$$
 \frac{1}{m!!}   |S_m|  = (m-1) !! = |\mathcal{P}_m|.
$$
Then prove that 
\[
 \Pf(A) = \sum_{\sigma \in \mathcal{P}_{m}} (-1)^{\sigma} a_{\sigma(1)\sigma(2)} \cdot a_{\sigma(3)\sigma(4)} \cdots  a_{\sigma(m-1)\sigma(m)}.
\]

\bigskip 

\textbf{3.} 
Use the Gauss--Bonnet formula to prove that if $M_1$ and $M_2$ are two closed manifolds, then $\chi(M_1 \times M_2) = \chi(M_1)  \chi(M_2)$. Extend this to the case where one of the manifolds has a non-empty boundary.

\bigskip 
\textbf{4.} 
Generalize the Gauss--Bonnet Formula to the case of a closed non-orientable manifold. Can this be done for manifolds with non-empty boundaries?

\bigskip 
\textbf{5.}
Verify equation \eqref{eq.homega} in the proof of Proposition  \ref{Primitive}.

\bigskip 
\textbf{6.}  Prove that if $N$ is a hypersurface in a Riemannian manifold $(M,g)$, then the Gauss Equation can be written as
$$
  \bar{R}  = R + \frac{1}{2}  g \owedge g,
$$
where $R$ is the curvature tensor of the ambient manifold $(M,g)$, $\bar{R}$ is the intrinsic curvature tensor of the submanifold $N$ and 
$h$ is the second fundamental form of $N$.

\smallskip

Two proofs (at least)  are possible. One proof is to reduce the above formula to the Gauss Equation expressed in a formalism familiar to the reader. The other is to try and write a direct proof using the moving frame formalism.

\bigskip 
\textbf{7.} 
Let $M$ be a 4-dimensional manifold. Use formula  \eqref{ChernPhik} to  calculate the forms $\Phi_0$, $\Phi_1$, $\Psi_0$, and $\Psi_1$.  Find then an explicit  primitive of the Pfaffian in this case (follow the steps in the proof of Proposition \ref{Primitive}).

\bigskip 
\textbf{8.} Our proof of the Gauss--Bonnet--Chern Formula uses the Poincaré--Hopf theorem. Show that, in fact, these two results are equivalent.

The argument will proceed in two steps. First, reverse the proof of Theorem \ref{th.CGB} to demonstrate that the sum of the  indices of a vector field is an invariant of the manifold. Then, identify this invariant by selecting a well-chosen vector field (for example, the gradient of a Morse function).

\bigskip 
\textbf{9.} Let $h_1,h_2 \in \mathcal{D}^{1,1}(M)$ be two double forms of type $(1,1)$ and set $A = h_1\owedge h_2$. Show that if $h_1$ and $h_2$ are symmetric, then $A$ satisfies the \emph{first Bianchi identity:}  \index{Bianchi's identities}
$$
 A(X,Y;U,V)+A(Y,U;X,V)+A(U,X;Y,V) = 0.
$$

\bigskip 
\textbf{10.} 
With the help of a computer algebra system such as Maple, Mathematica, or SageMath, verify the validity of some, or all, of the following  identities appearing in the present  paper:
\eqref{Calcula1},  \eqref{Calculaa2}, \eqref{Calculbb},  \eqref{eq.cmk},  \eqref{SpSq}   and  \eqref {sumgamma}. 
After mastering this task, you should be able to navigate through the computational details of Example  \ref{rotmetric}.

\bigskip 
\textbf{11.} 
In  Example \ref{ExampleBS} we have  discussed the Gauss--Bonnet Formula for the manifold $M = B^{p+1}\times S^q$ assuming $p$ is even and $q$ is odd. What happens if,  instead,  $p$ is odd and $q$ is even?

 
\appendix

\section{Appendix: Some Background in Differential Geometry}

\label{sec.bacground} 

In this appendix, we review some of the geometric concepts employed in this paper and establish our notation. Additionally, we provide a detailed introduction to the moving frame method.

\subsection{Some Linear Algebra on the (co)tangent Space} 
\label{secLinAlg}

By its very definition, a Riemannian metric $g$ on a smooth manifold $M$ endows each tangent space with a Euclidean structure. The scalar product of two tangent vectors $X$ and $Y$ in $T_pM$ will be denoted interchangeably as $g_p(X,Y)$ or $\langle X, Y \rangle_p$, and the Riemannian norm will be represented by $\|X\|_p = \sqrt{\langle X, Y \rangle_p}$. The subscript $p$ will be omitted when there is no ambiguity. The same notation will be used when $X$ and $Y$ are smooth vector fields over $M$; in this case, the scalar product becomes a smooth function on $M$.

\medskip 

To any tangent vector $X\in T_pM$, we associate a covector $X^\flat \in T_p^*M$ defined by the relation $X^\flat (Y) = \bra X, Y\ket$ for any $Y\in T_pM$. This yields a linear isomorphism from the tangent space $T_pM$ to the cotangent space $T_p^*M$, and the inverse isomorphism maps a covector  $\varphi \in T_p^*M$ to the unique  tangent vector $\varphi^\sharp \in T_pM$ such that $\bra \varphi^\sharp, Y\ket = \varphi (Y)$, again for any $Y\in T_pM$. The operators $\sharp$ and $\flat$ are sometimes called the \textit{musical isomorphisms}.\footnote{The name and notation have apparently been introduced by Marcel Berger in the 1970s.} 
A scalar product is then defined on  $\Lambda^k T^*_pM$, first for $k=1$ by 
$$
  \bra \alpha, \beta \ket =  \bra \alpha^\sharp, \beta^\sharp \ket 
$$
for any $\alpha, \beta  \in T^*_pM$, then by the formula: 
\begin{equation}\label{scalprodext}
   \bra \alpha_1\wedge \cdots \wedge \alpha_k, \beta_1\wedge \cdots \wedge \beta_k, \ket  = \det \left(  \bra \alpha_i, \beta_j \ket \right),
\end{equation}
where $\alpha_i, \beta_j \in T^*_pM$. The ``volume element'' at the point $p$ is the  unique element $\Theta = \Theta_p \in \Lambda^m T_p^*M$ representing the natural orientation and such that $\|\Theta\| = 1$. Its existence and uniqueness follows from the fact that 
$\dim \left( \Lambda^m T_p^*M\right)  = 1$. It  is characterized by the condition 
$$ 
 \Theta_p (e_1, \dots, e_m) = 1,
$$
for any positively oriented orthonormal basis $\{e_1, \dots, e_m\}$ of $T_pM$. 
The volume element is independent of the chosen moving frame, making it a globally defined $m$-form on the manifold. We will commonly use the notation:
$$
 \dvol_g = \Theta.
$$

This notation, however, does not assume that $\Theta$ is a closed form. Instead, we view it as a measure on the manifold. In local coordinates, it can be expressed as:
$$
 \dvol_g = \sqrt{\det (g_{ij})} \,  dx_1\cdots dx_m,
$$
where the $g_{ij}$ are the components of the metric tensor. 

\medskip 

We  recall that the \textit{Hodge star operator} is the isomorphism: $* : \Lambda^k T^*_pM \to \Lambda^{m-k} T^*_pM$ which is defined by the condition 
$$
  \alpha \wedge (*\alpha) = \|\alpha\|^2 \Theta.
$$
The dual basis of a positively oriented  orthonormal basis  $\{e_1, \dots, e_m\} \subset T_pM$ will be denoted by $\{\theta^1, \dots,\theta^m\} \subset T_p^*M$ of $T_pM$. It is given by  $\theta^i = e_i^\flat$ and we have 
$$
 \Theta = \theta^1\wedge \cdots \wedge \theta^m.
$$
It is easy to check that 
$$  
 * \theta^{i_1}\wedge  \cdots \wedge \theta^{i_k} = \varepsilon \;  \theta^{j_1}\wedge  \cdots \wedge \theta^{j_{n-k}}, 
$$
where $\{i_1, \dots, i_k, j_1, \dots, j_{m-k}\}$ is any permutation of $\{1, \dots, m\}$ and $\varepsilon=\pm 1$ is its signature. 
Furthermore, the collection
$$
  \{\theta^{i_1}\wedge  \cdots \wedge \theta^{i_k}\}_{i_1< i_2 < \cdots <i_k}
$$
forms an orthonormal basis of $\Lambda^k T^*_pM$.
Using these operators, we define the isomorphism $\varpi : T_pM \to \Lambda^{n-1} T^*_pM$ by\footnote{Using the interior product, one can also define it as $ \varpi_{_X} = \iota_{X} (\Theta)$.}
\begin{equation}\label{def.varpi}
 \varpi_{_X} = *(X^\flat).
\end{equation}
This isomorphism  will play an important role in the sequel; note that $\varpi_{_X}$ is also characterized by the condition
$$
  X^\flat \wedge   \varpi_{_X} = \|X\|^2 \, \Theta.
$$
As an important example, observe that if $\{e_1, \dots, e_m\}$ is a  positively oriented orthonormal basis of $T_pM$, then 
\begin{equation}\label{normalvarpi}
   \varpi_{{e_m}}  = (-1)^{m-1}\, \theta^{1}\wedge  \cdots \wedge \theta^{m-1}.
\end{equation}

\subsection{Connection and Curvature}

Riemannian geometry is of course much more than linear algebra in tangent spaces. To compare nearby tangent spaces effectively, we need the concept of  connection. In Kozsul's formalism, the Levi-Civita connection of  $(M,g)$  is  seen as the unique operator that associates to any pair of smooth vector fields $X,Y$ on $M$ a new vector field, denoted as $\nabla_XY$, and   satisfying the following conditions: \index{Levi-Civita connection}
\begin{enumerate}[(i)]
\item $\nabla_X(Y_1+Y_2) = \nabla_XY_1 + \nabla_XY_2$ \  and \ $\nabla_{X_1+X_2}Y = \nabla_{X_1}Y + \nabla_{X_2}Y$.
\item $\nabla_{fX}Y = f \nabla_XY$  \ and \  $\nabla_X(fY) = f \nabla_XY + X(f)Y$ for any $f \in C^{\infty}(M)$.
\item $\nabla_XY -   \nabla_YX = [X,Y]$ \  (the Lie bracket of $X$ and $Y$).
\item $Z \bra X,Y\ket  = \bra\nabla_ZX,Y\ket + \bra X,\nabla_ZY\ket$.
\end{enumerate}
We will denote by $R$ the associated covariant curvature tensor. \index{Curvature Tensor} It is the $(0,4)$-tensor field defined by
$$
   R(X,Y;Z,W) = \bra \nabla_{[X,Y]} Z + \nabla_Y\nabla_X Z - \nabla_X\nabla_Y Z,W\ket.
$$
Its basic properties are:   
\begin{enumerate}[(a)]
\item   $R$ is $C^{\infty}(M)$-multilinear.
\item   $R(X,Y,Z,W) = R(Z,W, X,Y) = - R(Y,X,Z,W) = - R(X,Y,W,Z)$.
\item   $R(X,Y,Z,W) + R(Y,Z,X,W) + R(Z,X,Y,W) = 0$ (the first Bianchi Identity).   \index{Bianchi's identities}
\end{enumerate}
A consequence of property (a) is that the value of $R(X,Y,Z,W)$ at a point $p\in M$ only depends on the values of the vectors
$X_p,Y_P,Z_p,W_p\in T_pM$ at this point  and not on the global, or even local, behavior of the vector fields. A consequence of (b) is that if 
$v,w \in T_pM$ are linearly independent, then the scalar 
$$
  K(v,w) = \frac{R(v,w;v,w)}{\|v\|^2\|w\|^2 - \bra u,v\ket^2}
$$
only depends on the $2$-plane $\sigma =\mathrm{Span}(u,v) \subset T_pM$ spanned by these vectors. This value is the \textit{sectional curvature} of $\sigma$.

\subsection{Moving the Frame with Élie Cartan}

The calculus of differential forms was developed in the 1920s by Élie Cartan, who employed it as a convenient formalism in Riemannian geometry.\footnote{Chern became acquainted with Cartan's calculus first through his interactions with E. Kähler in Hamburg, and later when he met Cartan in Paris. See his comments on Cartan in \cite{Interview}.}  Helpful references include Chapter 7 in \cite[Vol II]{Spivak}, Chapter 4 in T.J. Willmore's book \cite{Willmore}, and the books \cite{Cartan1925,Cartan} by Cartan himself.

\medskip 

Here we briefly recall the basic  definitions and we state and prove the main equations that are needed in this paper.  A \textit{moving frame} in some open set $U$ of the Riemannian manifold $(M,g)$ is simply a collection of $m$ smooth vector fields $e_1, \dots , e_m : U \to TU$ defined on $U$  such that $\bra e_i, e_j\ket = \delta_{ij}$  at every point of $U$. We will always assume the moving frame to be positively oriented. The  dual coframe $\theta^1, \dots \theta^m$  will be denoted as $\theta^i = e_i^\flat$. It  is defined by 
$$
 \theta^i(X) = g(e_i, X).
$$
\begin{definition} \label{defomegai}
The \emph{connection forms} associated with the moving frame $\{e_i\}$ are the $1$-forms defined on $U$  by:
$$
 \omega_{ij}(X) = g(e_i,\nabla_X e_j)= \theta^i(\nabla_X e_j),
$$
and the \emph{curvature forms} are the $2$-forms: 
$$
  \Omega_{ij}(X,Y) =   R(X,Y;e_i,e_j).
$$
\end{definition}
Observe that the connection forms are characterized by the identity 
$$
 \nabla_X e_j = \sum_{i=1}^m\omega_{ij}(X) e_i,
$$
note  also the skew symmetry: 
$$
  \omega_{ij}+ \omega_{ji} =   \Omega_{ij}+ \Omega_{ji} = 0.
$$

The next lemma gives us three useful formulas representing the curvature in the moving frame formalism:
\begin{lemma} \label{lem:formule_R}
Under the previously established notation for the moving coframe, the curvature forms $\Omega_{ij}$  can ve written as 
\begin{equation*}
\Omega_{ij} = \frac{1}{2} \sum_{k,l=1}^m R_{ijkl} \, \theta^k \wedge \theta^l.
\end{equation*}
And the curvature tensor $R$ is expressed as
\begin{align*}
R &= \frac{1}{4} \sum_{i,j,k,l=1}^m R_{ijkl} \, \theta^i \wedge \theta^j \otimes \theta^k \wedge \theta^l \\
&= \frac{1}{2} \sum_{i,j=1}^m \Omega_{ij} \otimes \theta^i \wedge \theta^j,
\end{align*}
where $R_{ijkl} = R(e_i,e_j,e_k,e_l)$.
\end{lemma}

\begin{proof}
From the symmetries of the curvature tensor we have $R_{klij} = R_{ijkl} =  -R_{ijlk}$,
therefore
$$
 \Omega_{ij} =   \sum_{k,l = 1}^m R_{klij} \, \theta^k \otimes \theta^l 
 =   \sum_{k,l = 1}^m R_{ijkl} \, \theta^k \otimes \theta^l 
  = \frac{1}{2} \sum_{k,l = 1}^m R_{ijkl} \, \theta^k \wedge \theta^l.
$$
This proves the first equation. The second equation is proved as follows
\begin{align*}
R  & =  \sum_{i,j,k,l = 1}^m R_{ijkl}  \left( \theta^i \otimes \theta^j\right) \otimes\left( \theta^k\otimes \theta^l \right) 
=   \sum_{i,j,k,l = 1}^m R_{ijkl} \left( \theta^k\otimes \theta^l \right) \otimes \left( \theta^i \otimes \theta^j\right) 
\\&= \frac{1}{4}  \sum_{i,j,k,l = 1}^m  R_{ijkl}\left( \theta^k\otimes \theta^l -\theta^l\otimes \theta^k\right) \otimes \left( \theta^i \otimes \theta^j- \theta^j\otimes \theta^i\right) 
\\ &= \frac{1}{4}  \sum_{i,j,k,l = 1}^m R_{ijkl}\left( \theta^k\wedge\theta^l \right) \otimes \left( \theta^i \wedge\theta^j\right) 
\\ &=  \frac{1}{2} \sum_{k,l = 1}^m \Omega_{ij} \otimes \theta^i\wedge\theta^j.
\end{align*}
\end{proof}

\begin{proposition}[{\'E}. Cartan's Structure Equations] 
The exterior derivative of the moving coframe and the connection forms are given by 
\begin{equation}\label{StructuresEq}
 d\theta^i =  -\sum_{j=1}^m\omega_{ij}\wedge \theta^j \quad  \mathrm{and} \quad  d\omega_{ij} = \Omega_{ij} 
- \sum_{k=1}^m \omega_{ik} \wedge \omega_{kj}.
\end{equation}
\end{proposition}

\medskip 

\textbf{Proof.}
We prove the first structure equation as follows:
\begin{align*}
 d\theta^i(X,Y) &= X (\theta^i(Y)) - Y (\theta^i(X)) - \theta^i ([X,Y])
\\ &= X  g(e_i, Y) - Y  g(e_i, X) -  g(e_i,[X,Y])
\\ &=  g(\nabla_X e_i, Y) + g(e_i, \nabla_XY)  - g(\nabla_Y e_i, X) - g(e_i, \nabla_YX) -  g(e_i,[X,Y])
\\ &=  g(\nabla_X e_i, Y) - g(\nabla_Y e_i, X)
\\&= \sum_{j= 1}^m \left( g(\nabla_X e_i, e_j)g(e_j, Y)  -  g(\nabla_Y e_i,e_j) g(e_j, X) \right) 
\\&= \sum_{j= 1}^m \left(\omega_{ji}(X)\theta^j(Y) - \omega_{ji}(Y)\theta^j(X)\right). 
\\&= - \sum_{j= 1}^m  \left(\omega_{ij}\wedge \theta^j\right) (X,Y)
\end{align*}
(we have used the symmetry of  $\nabla$ on the fourth line). 
To prove the second structure equation, we compute 
\begin{align*}
 d\omega_{ij} (X,Y) &= X (\omega_{ij}(Y)) - Y( \omega_{ij}(X)) -\omega_{ij}([X,Y])
\\&= X(g(e_i,\nabla_Y e_j)) - Y(g(e_i,\nabla_X e_j)) - g( e_i,\nabla_{[X,Y]}e_j)
\\&= (g(\nabla_X e_i,\nabla_Y e_j) + g(e_i,\nabla_X \nabla_Y e_j)) - (g(\nabla_Y e_i,\nabla_Xe_j) - g(e_i,\nabla_Y \nabla_Xe_j)) -  g( e_i,\nabla_{[X,Y]}e_j)
\\&=  g(\nabla_X e_i,\nabla_Y e_j) - g(\nabla_Y e_i,\nabla_X e_j)  +  g( e_i, \nabla_X \nabla_Y e_j - \nabla_Y\nabla X e_j -\nabla_{[X,Y]} e_j) 
\\ &= g(\nabla_X e_i,\nabla_Y e_j) -  g(\nabla_Y e_i,\nabla_X e_j)  -  R(X,Y;e_j,e_i). 
\end{align*}

Observe now that on the one  hand  
$$
 g(\nabla_X e_i,\nabla_Ye_j) =  \sum_{k=1}^m \theta^k(\nabla_Xe_i) \theta^k(\nabla_Ye_j) = \sum_{k=1}^m \omega_{k,i}(X)\omega_{k,j}(Y) = - \sum_{k=1}^m  \omega_{i,k}(X)\omega_{k,j}(Y),
$$
therefore
$$
 g(\nabla_X e_i,\nabla_Y e_j) -  g(\nabla_Y e_i,\nabla_X e_j)    =  -\sum_{k=1}^m \omega_{ik} \wedge \omega_{kj} (X,Y).
$$
On the other hand  we have 
$$
  d\omega_{ij} (X,Y) =  \Omega_{ij} (X,Y) -  \sum_{k=1}^m \omega_{ik} \wedge \omega_{kj} (X,Y),
$$
Because \ $-R(X,Y;e_j,e_i)  =  R(X,Y;e_i,e_j) = \Omega_{ij} (X,Y)$.
\qed

\medskip 
 
\begin{lemma}[The Bianchi Identities]   \index{Bianchi's identities}
The curvature forms also satisfy  the following identities:
\begin{equation}\label{Bianchi1et2}
  \sum_{j=1}^m \Omega_{ij}  \wedge \theta^j  = 0 \quad  \mathrm{and} \quad  
  d\Omega_{ij}  =  \sum_{k=1}^m \left(\Omega_{ik}\wedge \omega_{kj} -\omega_{ik} \wedge  \Omega_{kj}\right).
\end{equation}
\end{lemma}
\begin{proof}
Since $d^2\theta^i = 0$ and $ d^2\omega_{ij} = 0$, we have for $1 \leq i \leq m$, 
\begin{align*}
 0 &= d \left(\sum_{j=1}^m\omega_{ij}\wedge \theta^j\right)  = 
\sum_{j=1}^md\omega_{ij}\wedge \theta^j - \sum_{j=1}^m\omega_{ij}\wedge d\theta^j  
\\ &= 
 \sum_{j=1}^m \left(\Omega_{ij} 
-  \sum_{k=1}^m \omega_{ik} \wedge \omega_{kj}\right) \wedge \theta^j + \sum_{j=1}^m\omega_{ij}\wedge  \left(\sum_{k=1}^m\omega_{jk}\wedge \theta^k\right) 
\\ & =
 \sum_{j=1}^m \Omega_{ij}  \wedge \theta^j + \sum_{j,k=1}^m \left(-\omega_{ik} \wedge \omega_{kj} \wedge \theta^j  + \omega_{ij}\wedge\omega_{jk}\wedge \theta^k \right) 
\\ & =
 \sum_{j=1}^m \Omega_{ij}  \wedge \theta^j,
\end{align*}
which proves the first equation. To prove the second equation, we compute
\begin{align*}
0 &= d^2 \omega_{ij} =  d\left(\Omega_{ij} -  \sum_{k=1}^m \omega_{ik} \wedge \omega_{kj}\right) 
=  d \Omega_{ij} -  \sum_{k=1}^m d\omega_{ik} \wedge \omega_{kj}    +  \sum_{k=1}^m \omega_{ik} \wedge d\omega_{kj} 
\\ &=  d \Omega_{ij} -  \sum_{k=1}^m \left(\Omega_{ik} -  \sum_{s=1}^m \omega_{is} \wedge \omega_{sk}\right) \wedge \omega_{kj}   +
 \sum_{k=1}^m \omega_{ik} \wedge  \left(\Omega_{kj} - \sum_{s=1}^m \omega_{ks} \wedge \omega_{sj}\right)  
\\& =
d \Omega_{ij} -  \sum_{k=1}^m \left(\Omega_{ik}\wedge \omega_{kj} -\omega_{ik} \wedge  \Omega_{kj}\right) 
+  \sum_{k,s=1}^m \left(\omega_{is} \wedge \omega_{sk}\wedge \omega_{kj}  - \omega_{ik} \wedge  \omega_{ks} \wedge \omega_{sj}\right) 
\\&= d \Omega_{ij} -  \sum_{k=1}^m \left(\Omega_{ik}\wedge \omega_{kj} -\omega_{ik} \wedge  \Omega_{kj}\right).  
\end{align*}
\end{proof}

We conclude this section by discussing the effect  of changing the  moving frame. Suppose that $e_1', \dots,  e_m'$ is another orthonormal moving frame defined on $U$; then we have 
$$
  e_j' = \sum_{i=1}^m a_{ij} e_i,  
$$
where $a = (a_{ij})$ is a smooth function defined on $U$ with values in the orthogonal group $O(m)$. The corresponding moving coframe is 
$$
 \theta'^j  = \sum_{i=1}^m a_{ij} \theta^i
$$
(indeed we have $\theta'^j(X) = g\left(e'_j, X \right)  = \sum_{i=1}^m a_{ij}g (e_i,X) = \sum_{i=1}^m a_{ij} \theta^i(X)$).  About  the connection and curvature forms, we have the
\begin{lemma}\label{lem.chgframeomega}
Under a change of moving frame, the connection forms transform as
$$
  \omega'_{ij} = \sum_{r,s=1}^m a_{si}a_{rj} \omega_{sr}   + \sum_{r=1}^m a_{ri}   da_{rj},
$$
and the curvature  forms transform as
$$
  \Omega'_{ij} = \sum_{r,s=1}^m a_{si}a_{rj} \Omega_{sr}  .
$$
\end{lemma}

\textbf{Remark.} In this lemma we don't need to assume that our moving frames are positively oriented.  In matrix notation, the transformation formula for the connection can be written as
$$
  \omega' =  a^{-1}\omega a + a^{-1}da
$$ 
(compare with \cite[vol. 2, page 280]{Spivak}). This is the standard \textit{gauge transformation formula} for connections in a principal bundle. Note however that in our case the inverse matrix $a^{-1}$ is simply its transpose.

\begin{proof}
Recall first that 
$$
 \nabla_X e_j' =  \nabla_X e_j' =   \nabla_X\left( \sum_{r=1}^m a_{rj}  e_r \right)  =  
\sum_{r=1}^m a_{rj} \nabla_X e_r  +  \sum_{r=1}^m da_{rj}(X)  e_r.
$$
Therefore 
\begin{align*}
  \omega'_{ij} (X) &=  \theta'^i\left(\nabla_X e'_j \right) = \sum_{r,s=1}^m a_{si}a_{rj} \left( \theta^s\left( \nabla_X e_r\right)  +   da_{rj}(X) \theta^s( e_r) \right) 
\\&=
\sum_{r,s=1}^m a_{si}a_{rj} \omega_{sr}(X)  + \sum_{r=1}^m a_{ri}   da_{rj}(X).
\end{align*}
 This proves the first formula. The proof of the second formula is immediate from the definition: 
$$
   \Omega'_{ij}(X,Y) =   R(X,Y;e'_i,e'_j) = \sum_{r,s=1}^m a_{si}a_{rj}  R(X,Y;e_r,e_s)  = 
   \sum_{r,s=1}^m a_{si}a_{rj}\Omega_{sr}(X,Y) ,
$$
because of the tensorial nature of $R$.
\end{proof}


\bigskip

\paragraph{Acknowledgement.}  The author extends his gratitude to Niky Kamran for his encouragements and meticulous review of the manuscript, as well as to Ruth Kellerhals for her comments on the case of space forms and  for pointing to  the precise reference in Ratcliffe's  treatise. 
We also acknowledge  Adrien Marcone for the valuable discussions we had regarding the Gauss--Bonnet Chern formula; his thesis \cite{Marcone} has  been a helpful resource in preparing the current paper.
Special thanks are also due to Jean-Pierre Bourguignon for his comments and  for drawing our attention to the notable paper \cite{Lanczos} by C. Lanczos.

\end{document}